\documentclass[11pt,a4paper, leqno]{article}%
\usepackage[centertags]{amsmath}
\usepackage{amsfonts}
\usepackage{amssymb}
\usepackage{amsthm}
\usepackage{epsfig}
\usepackage{setspace}
\usepackage{comment}
\usepackage{ae}
\usepackage[usenames]{color}
\usepackage{graphicx}%
\providecommand{\U}[1]{\protect\rule{.1in}{.1in}}
\theoremstyle{plain}
\newtheorem{thm}{Theorem}[section]
\newtheorem{lem}[thm]{Lemma}
\newtheorem{cor}[thm]{Corollary}
\newtheorem{prop}[thm]{Proposition}
\newtheorem{prob}[thm]{Problem}
\theoremstyle{definition}
\newtheorem{defi}[thm]{Definition}

\theoremstyle{remark}
\newtheorem{exmp}[thm]{Example}
\newtheorem{rem}[thm]{Remark}

\renewcommand{\epsilon}{\varepsilon}
\renewcommand{\P}{\operatorname{P}}

\newcommand{\E}{\mathbb{E}}
\renewcommand{\Re}{\mathsf{Re}}
\renewcommand{\Im}{\mathsf{Im}}

\renewcommand{\i}{\mathrm{i}}

\begin{document}

\title{Statistical Skorohod embedding problem and its generalizations}
\author{Denis Belomestny$^{1} $ and John Schoenmakers$^{2}$}
\maketitle

\begin{abstract}
Given a L\'evy process $L$, we consider the so-called statistical Skorohod
embedding problem of recovering the distribution of an independent random time
$T$ based on i.i.d. sample from $L_{T}.$ Our approach is based on the genuine
use of the Mellin and Laplace transforms. We propose a consistent estimator
for the density of $T,$ derive its convergence rates and prove their
optimality. It turns out that the convergence rates heavily depend on the
decay of the Mellin transform of $T.$ We also consider the application of our
results to the problem of statistical inference for variance-mean mixture
models and for time-changed L\'evy processes.

\end{abstract}

\footnotetext[1]{D. Belomestny, Duisburg-Essen University and National
Research University Higher School of Economics, Thea-Leymann-Str. 9, D-45127
Essen, Germany, \texttt{denis.belomestny@uni-due.de}
\par
Partially supported by the Deutsche Forschungsgemeinschaft through the SFB 823
``Statistical modeling of nonlinear dynamic processes'' and by Laboratory for
Structural Methods of Data Analysis in Predictive Modeling, MIPT, RF
government grant, ag. 11.G34.31.0073.} \footnotetext[2]{Weierstrass Institute
for Applied Analysis and Stochastics, Mohrenstr. 39, 10117 Berlin, Germany,
\texttt{schoenma@wias-berlin.de}} \emph{Keywords:} Skorohod embedding problem,
L\'evy process, Mellin transform, Laplace transform, variance mixture models,
time-changed L\'evy processes.

\section{Introduction}

The so called Skorohod embedding (SE) problem or Skorohod stopping problem was
first stated and solved by Skorohod in 1961. This problem can be formulated as follows.

\begin{prob}
[Skorohod Embedding Problem]For a given probability measure $\mu$ on
$\mathbb{R},$ such that $\int|x| d\mu(x)<\infty$ and $\int x d\mu(x)=0,$ find
a stopping time $T$ such that $B_{T}\sim\mu$ and $B_{T\wedge t}$ is a
uniformly integrable martingale.
\end{prob}

The SE problem has recently drawn much attention in the literature, see e.g.
Ob{\l }{\'o}j, \cite{obloj2004skorokhod}, where the list of references
consists of more than 100 items. In fact, there is no unique solution to the
SE problem and there are currently more than $20$ different solutions
available. This means that from a statistical point of view, the SE problem is
not well posed. In this paper we first study what we call \textit{statistical
Skorohod embedding} (SSE) problem.

\begin{prob}
[Statistical Skorohod Embedding Problem]\label{stat_skorohod} Based on i.i.d.
sample $X_{1},\ldots, X_{n}$ from the distribution of $B_{T}$ consistently
estimate the distribution of the random time $T\geq0,$ where $B$ and $T$ are
assumed to be independent.
\end{prob}

The independence of $B$ and $T$ is needed to ensure the identifiability of the
distribution of $T$ from the distribution of $B_T$. It is shown that the SSE
problem is closely related to the multiplicative deconvolution problem. Using
the Mellin transform technique, we construct a consistent estimator for the
density of $T$ and derive its convergence rates in different norms.
Furthermore, we show that the obtained rates are optimal in minimax sense. The
asymptotic normality of the proposed estimator is addressed as well. Next, we
generalize the SSE problem by replacing the standard Brownian motion with a
general L\'evy process. The generalized SSE problem turns out to be much more
involved and its solution requires some new ideas. Using a genuine combination
of the Laplace and Mellin transforms, we construct a consistent estimator,
derive its minimax convergence rates and prove that these rates basically
coincide with the rates in the SSE problem.

Some particular cases of generalized statistical Skorohod embedding problem
have been already studied in the literature. For example, the case of the
stopped Poisson process was considered in the recent paper of Comte and
Genon-Catalot, \cite{comteadaptive}.

\section{Statistical Skorohod embedding problem}

Let $B$ be a Brownian motion and let a random variable $T\geq0$ be independent
of $B.$ We then have,%
\begin{equation}
X:=B_{T}\sim\sqrt{T}\,B_{1} \label{scaling_prop_bm}%
\end{equation}
and the problem of reconstructing $T$ is related to a multiplicative
deconvolution problem. While for additive deconvolution problems the Fourier
transform plays an important role, here we can conveniently use the Mellin transform.

\begin{defi}
Let $\xi$ be a non-negative random variable with a probability density
$p_{\xi}$, then the \textit{Mellin transform} of $p_{\xi}$ is defined via
\begin{equation}
\mathcal{M}[p_{\xi}](z):=\mathbb{E}[\xi^{z-1}]=\int_{0}^{\infty}p_{\xi
}(x)x^{z-1}\,dx \label{def}%
\end{equation}
for all $z\in\mathcal{S}_{\xi}$ with $\mathcal{S}_{\xi}=\bigl\{z\in
\mathbb{C}:\mathbb{E}[\xi^{\Re z-1}]<\infty\bigr\}.$
\end{defi}

Since $p_{\xi}$ is a density, it is integrable and so at least $\left\{
z\in\mathbb{C}:\Re(z)=1\right\}  \subset\mathcal{S}_{\xi}.$ Under mild
assumptions on the growth of $p_{\xi}$ near the origin, one obtains
\[
\left\{  z\in\mathbb{C}:0\leq a_{\xi}<\Re(z)<b_{\xi}\right\}  \subset
\mathcal{S}_{\xi}%
\]
for some $0\leq a_{\xi}<1\leq b_{\xi}.$ Then the Mellin transform (\ref{def})
exists and is analytic in the strip $a_{\xi}<\operatorname{Re}z<b_{\xi}.$ For
example, if $p_{\xi}$ is essentially bounded in a right-hand neighborhood of
zero, we may take $a_{\xi}=0.$ The role of the Mellin transform in probability
theory is mainly related to the product of independent random variables: in
fact it is well-known that the probability density of the product of two
independent random variables is given by the Mellin convolution of the two
corresponding densities. Due to \eqref{scaling_prop_bm}, the SSE problem is
closely connected to the Mellin convolution. Suppose that the random time $T$
has a density $p_{T}$ and that we may take $0\leq a_{T}<1\leq b_{T}.$ Since
$\mathcal{S}_{|B_{1}|}\supset\left\{  z\in\mathbb{C}:\Re(z)>0\right\}  ,$ we
derive for $\max(2a_{T}-1,0)<\Re(z)<2b_{T}-1,$
\begin{multline*}
\mathcal{M}[p_{|X|}](z)=\mathbb{E}\bigl[|B_{1}|^{z-1}\bigr]\mathbb{E}%
\bigl[T^{(z-1)/2}\bigr]\\
=\mathcal{M}[p_{|B_{1}|}](z)\mathcal{M}[p_{T}]((z+1)/2)=\frac{2^{(z-1)/2}%
}{\sqrt{\pi}}\Gamma(z/2)\mathcal{M}[p_{T}]((z+1)/2).
\end{multline*}
As a result
\[
\mathcal{M}[p_{T}](z)=\frac{\sqrt{\pi}}{2^{z-1}}\frac{\mathcal{M}%
[p_{|X|}](2z-1)}{\Gamma(z-1/2)},\text{ \ \ }\max(a_{T},1/2)<\Re(z)<b_{T}%
\]
and the Mellin inversion formula yields
\begin{align}
p_{T}(x)  &  =\frac{1}{2\pi}\int_{\gamma-{\i}\infty}^{\gamma+{\i}\infty
}x^{-\gamma-{\i} v}\mathcal{M}[p_{T}](\gamma+{\i} v)\,dv\text{ }%
\label{inverse_mellin}\\
&  =\frac{1}{\sqrt{\pi}}\int_{-\infty}^{\infty}x^{-\gamma-{\i} v}%
\frac{\mathcal{M}[p_{|X|}](2\left(  \gamma+{\i} v\right)  -1)}{2^{\gamma+{\i}
v}\Gamma(\gamma+{\i} v-1/2)}\,dv\text{\ \ for \ }\max(a_{T},1/2)<\gamma
<b_{T},\text{ \ \ }x>0.\nonumber
\end{align}
Furthermore, the Mellin transform of $p_{|X|}$ can be directly estimated from
the data $X_{1},\ldots,X_{n}$ via the empirical Mellin transform:
\begin{equation}
\mathcal{M}_{n}[p_{|X|}](z):=\frac{1}{n}\sum_{k=1}^{n}|X_{k}|^{z-1},\text{
\ \ }\Re(z)>1/2, \label{melest}%
\end{equation}
where the condition $\Re(z)>1/2$ guarantees that the variance of the estimator
(\ref{melest}) is finite$.$ Note however that the integral in
\eqref{inverse_mellin} may fail to exist if we replace $\mathcal{M}[p_{|X|}]$
by $\mathcal{M}_{n}[p_{|X|}].$ We so need to regularize the inverse Mellin
operator. To this end, let us consider a kernel $K(\cdot)\geq0$ supported on
$[-1,1]$ and a sequence of bandwidths $h_{n}$ $>0$ tending to $0$ as
$n\rightarrow\infty.$ Then we define, in view of (\ref{melest}), for some
$\max(a_{T},3/4)<\gamma<b_{T},$
\begin{equation}
\label{p_Tn_bm}p_{T,n}(x):=\frac{1}{\sqrt{\pi}}\int_{-\infty}^{\infty
}x^{-\gamma-{\i} v}K(vh_{n})\frac{\mathcal{M}_{n}[p_{|X|}](2(\gamma+{\i}
v)-1)}{2^{\gamma+{\i} v}\Gamma(\gamma-1/2+{\i} v)}\,dv.
\end{equation}
For our convergence analysis, we will henceforth take the simplest kernel
\[
K(y)=1_{[-1,1]}(y),
\]
but note that in principle other kernels may be considered as well. The next
theorem states that $p_{T,n}$ converges to $p_{T}$ at a polynomial rate,
provided the Mellin transform of $p_{T}$ decays exponentially fast. We shall
use throughout the notation $A\lesssim B$ if $A$ is bounded by a constant
multiple of $B$, independently of the parameters involved, that is, in the
Landau notation $A=O(B)$.

\begin{thm}
\label{sep_conv_rates} For any $\beta>0,$ $\gamma>0$ and $L>0$, introduce the
class of functions
\[
\mathcal{C}(\beta,\gamma,L)=\left\{  f:\int_{-\infty}^{\infty}\left\vert
\mathcal{M}[f](\gamma+{\i} v)\right\vert e^{\beta\left\vert v\right\vert
}\,dv<L\right\}  .
\]
Assume that $p_{T}\in\mathcal{C}(\beta,\gamma,L)$ for some $\beta>0,$ $L>0$
and
\begin{equation}
\max((a_{T}+1)/2,3/4)<\gamma<b_{T}. \label{sg}%
\end{equation}
Then for some constant $C_{\gamma,L}$ depending on $\gamma$ and $L$ only, it
holds
\begin{equation}
\sup_{x\geq0}\mathbb{E}\Bigl[\bigl\{x^{\gamma}|p_{T}(x)-p_{T,n}(x)|\bigr\}^{2}%
\Bigr]\leq C_{\gamma,L}\times%
\begin{cases}
e^{-2\beta/h_{n}}+\frac{1}{n}h_{n}^{2\left(  \gamma-1\right)  }e^{\pi/h_{n}%
}, & \gamma<1,\\
e^{-2\beta/h_{n}}+\frac{1}{n} e^{\pi/h_{n}}, & \gamma\geq1.
\end{cases}
\label{ms}%
\end{equation}
By next choosing
\begin{equation}
h_{n}=
\begin{cases}
\frac{\pi+2\beta}{\log n-2(1-\gamma)\log\log n}, & \gamma<1,\\
(\pi+2\beta)/\log n, & \gamma\geq1,
\end{cases}
\label{hn}%
\end{equation}
we arrive at the rate
\begin{equation}
\label{rates_bm_pol}\sup_{x\geq0}\sqrt{\mathbb{E}\Bigl[\bigl\{x^{\gamma}%
|p_{T}(x)-p_{T,n}(x)|\bigr\}^{2}\Bigr]}\lesssim%
\begin{cases}
n^{-\frac{\beta}{\pi+2\beta}}\log^{\frac{2(1-\gamma)\beta}{\pi+2\beta}}n, &
\gamma<1,\\
n^{-\frac{\beta}{\pi+2\beta}}, & \gamma\geq1
\end{cases}
\end{equation}
as $n\rightarrow\infty.$
\end{thm}

With a little bit more effort one can prove the strong uniform convergence of
the estimate $p_{n,T}.$

\begin{thm}
\label{uniform_upper_bound} Under conditions of Theorem~\ref{sep_conv_rates}
and for $\gamma<1$
\begin{align*}
\sup_{p_{T}\in\mathcal{C}(\beta,\gamma,L)}\sup_{x\geq0}\bigl\{x^{\gamma
}|p_{T,n}(x)-p_{T}(x)|\bigr\}=O_{a.s.}\left(  n^{-\frac{\beta}{\pi+2\beta}%
}\log^{\frac{2(1-\gamma)\beta}{\pi+2\beta}}n\right)  .
\end{align*}

\end{thm}

Let us turn now to some examples.

\begin{exmp}
\label{exmp_gamma} Consider the class of Gamma densities
\begin{align*}
p_{T}(x;\alpha)=\frac{x^{\alpha-1}\cdot e^{-x}}{\Gamma(\alpha)}, \quad x\geq0
\end{align*}
for $\alpha>0.$ Since
\begin{align*}
\mathcal{M}[p_{T}](z)=\frac{\Gamma(z+\alpha-1)}{\Gamma(\alpha)},
\quad\mathsf{Re}(z)>0,
\end{align*}
we derive that $p_{T}\in\mathcal{C}(\beta,\gamma,L)$ for all $0<\beta<\pi/2,$
$\gamma>0$ and some $L=L(\beta,\gamma)$ due to the asymptotic properties of
the Gamma function (see Lemma~\ref{lemma_gamma_asymp} in Appendix). As a
result, Theorem~\ref{sep_conv_rates} implies
\begin{align*}
\sup_{x\geq0}\sqrt{\mathbb{E}\Bigl[\bigl\{x^{\gamma}|p_{T}(x)-p_{T,n}%
(x)|\bigr\}^{2}\Bigr]}\lesssim n^{-\rho},\quad n\rightarrow\infty
\end{align*}
for any $\rho<1/4,$ provided $\gamma\geq1.$
\end{exmp}

\begin{exmp}
Let us look at the family of densities
\begin{align*}
p_{T}(x;q)=\frac{q\sin(\pi/q)}{\pi}\frac{1}{1+x^{q}}, \quad q\geq2,\quad
x\geq0.
\end{align*}
We have
\begin{align*}
\mathcal{M}[p_{T}](z)=\frac{\sin(\pi/q)}{\sin(\pi z/q)},\quad0<\mathsf{Re}%
(z)<q.
\end{align*}
Therefore, $p_{T}\in\mathcal{C}(\beta,\gamma,L)$ for all $0<\beta<\pi/q,$
$\gamma>0$ and $L=L(\beta,\gamma),$ implying
\begin{align*}
\sup_{x\geq0}\sqrt{\mathbb{E}\Bigl[\bigl\{x^{\gamma}|p_{T}(x)-p_{T,n}%
(x)|\bigr\}^{2}\Bigr]}\lesssim n^{-\rho},\quad n\rightarrow\infty
\end{align*}
for any $\rho<1/(2+q),$ provided $\gamma\geq1.$
\end{exmp}

If $\mathcal{M}[p_{T}]$ decays polynomially fast, we get the following result.

\begin{thm}
\label{sep_log_rates} Consider the class of functions%
\[
\mathcal{D}(\beta,\gamma,L)=\left\{  f:\int_{-\infty}^{\infty}\left\vert
\mathcal{M}[f](\gamma+{\i} v)\right\vert (1+|v|^{\beta})\,dv<L\right\}  ,
\]
and assume that $p_{T}\in\mathcal{D}(\beta,\gamma,L)$ for some $\beta>0$ and
$L>0$ and $\gamma$ as in (\ref{sg}). Then for some constant $D_{\gamma,L},$ it
holds
\begin{equation}
\sup_{x\geq0}\mathbb{E}\Bigl[\bigl\{x^{\gamma}|p_{T}(x)-p_{T,n}(x)|\bigr\}^{2}%
\Bigr]\leq D_{\gamma,L}\times%
\begin{cases}
h_{n}^{2\beta}+\frac{1}{n}h_{n}^{2\left(  \gamma-1\right)  }e^{\pi/h_{n}}, &
\gamma<1,\\
h_{n}^{2\beta}+\frac{1}{n} e^{\pi/h_{n}}, & \gamma\geq1.
\end{cases}
\label{ms1}%
\end{equation}
By choosing
\begin{equation}
h_{n}=\frac{\pi}{\log n-2\left(  \beta+1-\gamma\right)  \log\log n},
\label{hn1}%
\end{equation}
if $\gamma<1$ and
\begin{equation}
h_{n}=\frac{\pi}{\log n-2\beta\log\log n} \label{hn2}%
\end{equation}
for $\gamma\geq1,$ we arrive at
\begin{equation}
\sup_{x\geq0}\sqrt{\mathbb{E}\Bigl[\bigl\{x^{\gamma}|p_{T}(x)-p_{T,n}%
(x)|\bigr\}^{2}\Bigr]}\lesssim\log^{-\beta}(n),\quad n\rightarrow\infty.
\label{arr1}%
\end{equation}

\end{thm}

\begin{rem}
Due to the relation
\begin{align*}
\mathcal{M}[p_{T}](\gamma+{\i} v)=\mathcal{F}[e^{\gamma\cdot}p_{T}(e^{\cdot
})](v),\quad a_{T}<\gamma<b_{T},
\end{align*}
the conditions $p_{T}\in\mathcal{C}(\beta,\gamma,L)$ and $p_{T}\in
\mathcal{D}(\beta,\gamma,L)$ are closely related to the smoothness properties
of the function $e^{\gamma x}p_{T}(e^{x}).$ For example, if $p_{T}%
\in\mathcal{C}(\beta,\gamma,L),$ then
\begin{align*}
\int_{-\infty}^{\infty}\left\vert \mathcal{F}[e^{\gamma\cdot}p_{T}(e^{\cdot
})](v)\right\vert e^{\beta\left\vert v\right\vert }\,dv<L
\end{align*}
and the function $e^{\gamma x}p_{T}(e^{x})$ is called supersmooth in this
case, see Meister \cite{meister2009deconvolution} for the discussion on
different smoothness classes in the context of the additive deconvolution problems.
\end{rem}

\begin{table}[ptb]
\begin{center}%
\begin{tabular}
[c]{|c|c|c|}\hline
\multicolumn{2}{|c|}{} & \\
\multicolumn{2}{|c|}{$\mathcal{C}(\beta,\gamma,L) $} & $\mathcal{D}%
(\beta,\gamma,L) $\\
\multicolumn{2}{|c|}{} & \\\hline\hline
$\gamma<1$ & $\gamma\geq1$ & \\\cline{1-2}
&  & \\
$n^{-\frac{\beta}{\pi+2\beta}}\log^{\frac{2(1-\gamma)\beta}{\pi+2\beta}}(n)$ &
$n^{-\frac{\beta}{\pi+2\beta}}$ & $\log^{-\beta}(n) $\\
&  & \\\hline
\end{tabular}
\end{center}
\caption{Minimax rates of convergence for the classes $\mathcal{C}%
(\beta,\gamma,L)$ and $\mathcal{D}(\beta,\gamma,L).$ }%
\label{RCS}%
\end{table}The rates of Theorem~\ref{sep_conv_rates} and
Theorem~\ref{sep_log_rates} summarized in Table~\ref{RCS} are in fact optimal
(up to a logarithmic factor) in minimax sense for the classes $\mathcal{C}%
(\beta,\gamma,L)$ and $\mathcal{D}(\beta,\gamma,L),$ respectively.

\begin{thm}
\label{low_bound} Fix some $\beta>1.$ There are $\varepsilon>0$ and $x>0$ such
that
\begin{align*}
&  \liminf_{n\to\infty}\inf_{p_{n}}\sup_{p_{T}\in\mathcal{C}(\beta,\gamma
,L)}\P ^{\otimes n}_{p_{T}}\Bigl(|p_{T}(x)-p_{n}(x)|\ge\varepsilon\,
n^{-\frac{\beta}{\pi+2\beta}}\log^{-\rho}(n)\Bigr)>0,\\
&  \liminf_{n\to\infty}\inf_{p_{n}}\sup_{p_{T}\in\mathcal{D}(\beta,\gamma
,L)}\P ^{\otimes n}_{p_{T}}\Bigl(|p_{T}(x)-p_{n}(x)|\ge\varepsilon\log
^{-\beta}(n)\Bigr)>0,
\end{align*}
for some $\rho>0,$ where the infimum is taken over all estimators (i.e. all
measurable functions of $X_{1},\ldots,X_{n}$) of $p_{T}$ and $\P ^{\otimes
n}_{p_{T}}$ is the distribution of the i.i.d. sample $X_{1},\ldots, X_{n}$
with $X_{1}\sim W_{T}$ and $T\sim p_{T}.$
\end{thm}

\subsection{Asymptotic normality}

In the case of $K(v)=1_{[-1,1]}(v),$ the estimate $p_{T,n}(x)$ can be written
as%
\begin{align*}
p_{T,n}(x)  &  :=\frac{1}{\sqrt{\pi}}\int_{-1/h_{n}}^{1/h_{n}}\left[ \frac
{1}{n}\sum_{k=1}^{n}|X_{k}|^{2(\gamma+\i v-1)}\right] \frac{x^{-\gamma-\i v}%
}{2^{\gamma+\i v}\Gamma(\gamma-1/2+\i v)}\,dv\\
&  =\frac{1}{n}\sum_{k=1}^{n}Z_{n,k}%
\end{align*}
where
\[
Z_{n,k}:=\frac{1}{\sqrt{\pi}}\int_{-1/h_{n}}^{1/h_{n}}|X_{k}|^{2(\gamma+\i
v-1)}\frac{x^{-\gamma-\i v}}{2^{\gamma+\i v}\Gamma(\gamma-1/2+\i v)}\,dv.
\]
The following theorem holds

\begin{thm}
\label{asymp_norm} Suppose that
\begin{align*}
\left. \frac{d}{du}\left( \Gamma(2\gamma-3/2+\i u)\mathcal{M}[p_{T}%
](2\gamma-1+\i u)\right) \right| _{u=0}\neq0,
\end{align*}
and
\begin{align*}
\int_{-\infty}^{\infty}\left\vert \mathcal{M}[p_{T}](2\gamma-1+\i
u)\right\vert du <\infty,
\end{align*}
then
\begin{align*}
\rho^{-1}_{n}\bigl(p_{T,n}(x)-\mathbb{E}[p_{T,n}(x)]\bigr)\overset
{\mathcal{D}}{\longrightarrow} \mathcal{N}(0,\sigma^{2})
\end{align*}
for some $\sigma^{2}>0,$ where $\rho_{n}=n^{-1/2} h_{n}^{2\left(
\gamma-1\right) }\log^{-2}\left( 1/h_{n}\right) \exp\left[ \pi/h_{n}\right](1+o(1))$
and $h_{n}\asymp c\log^{-1}(n)$ for some $c>0,$ as \(n\to \infty.\)
\end{thm}

\section{Generalised statistical Skorohod embedding problem}

\label{seq: gsse} In this section we generalize the statistical Skorohod
embedding problem to the case of L\'evy processes. In particular, we consider
the following problem.

\begin{prob}
\label{stat_skorohod_gen} Based on i.i.d. sample $X_{1},\ldots, X_{n}$ from
the distribution of $\mu,$ estimate the distribution of the random time
$T\geq0$ independent of a L\'evy process $L$ such that $L_{T}\sim\mu.$
\end{prob}

Note that the situation here is much more difficult than before, since the
L\'{e}vy processes do not have, in general, the scaling property
\eqref{scaling_prop_bm}. Hence the approach based on the Mellin deconvolution
technique can not be applied any longer. Let $(L_{t},\,t\geq0)$ be a L\'{e}vy
process with the triplet $(\mu,\sigma^{2},\nu).$ Define a curve in
$\mathbb{C}$
\[
\ell:=\Bigl\{\mathsf{Re}(\psi(u))+{\i}\,\mathsf{Im}(\psi(u)),\,u\in
\mathbb{R}_{+}\Bigr\},
\]
where $\psi(u)=-t^{-1}\log(\mathbb{E}(\exp(\i uL_{t}))).$ Our approach to
reconstruct the distribution of $T$ is based on the simple identity
\begin{align}
\label{cfx_main}\mathcal{F}[p_{X}](\lambda)=\mathbb{E}[\exp({\i}\lambda
L_{T})]=\mathcal{L}[p_{T}](\psi(\lambda)).
\end{align}
It is well known that the Laplace transform of $\mathcal{L}[p_{T}](u)$ is
analytic in the domain $\bigl\{\mathsf{Re}(u)>0\bigr\}.$ \begin{figure}[h]
\centering
\includegraphics[width=0.65\textwidth]{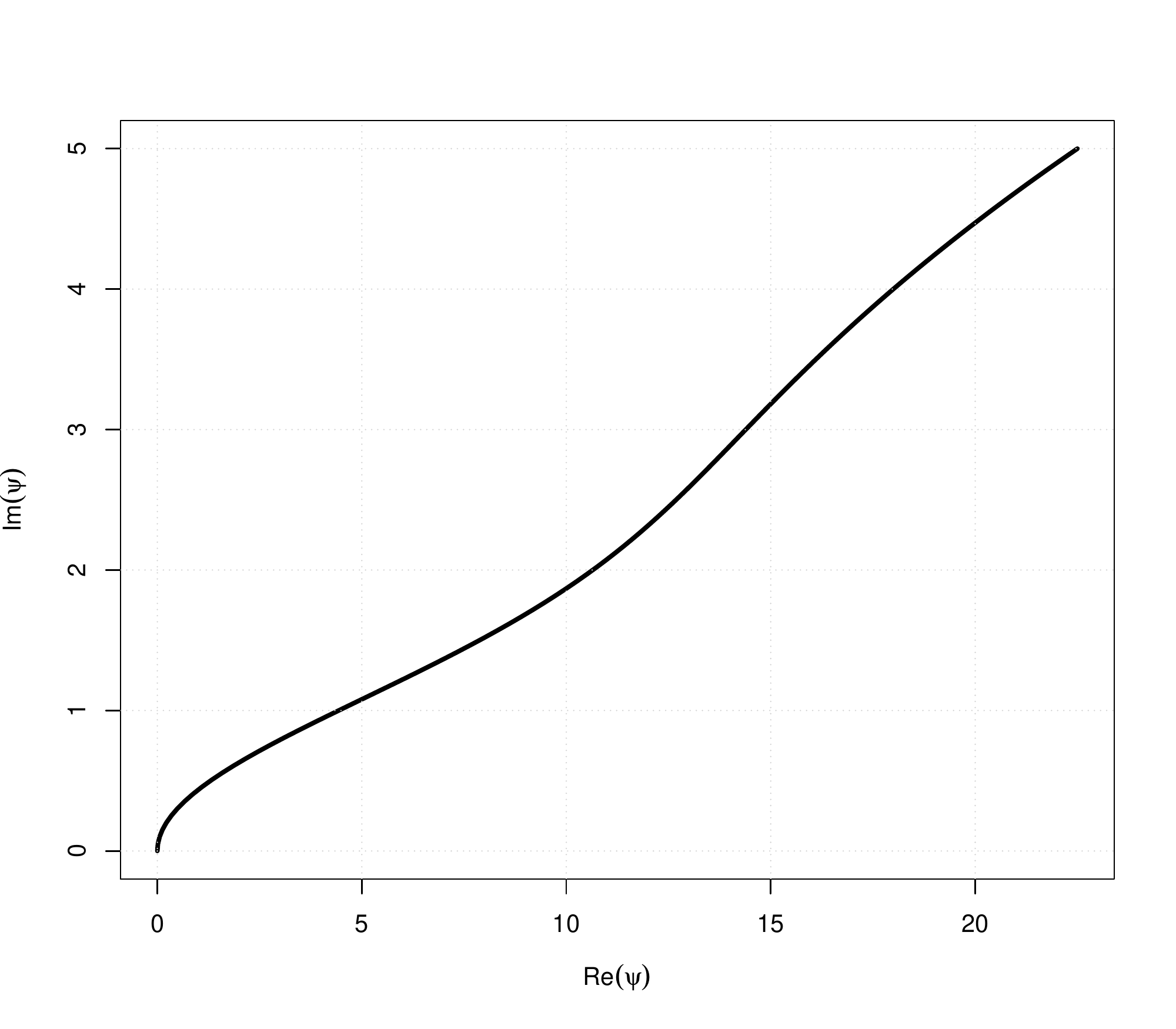}\caption{A typical shape
of the contour $\ell$.}%
\label{fig:ell}%
\end{figure}The following proposition shows that the object $\mathcal{M}%
[\mathcal{L}[p_{T}]](z)$ is well defined and that it can be related to the
Fourier transform of $p_{X},$ which in turn can be estimated from the data.

\begin{prop}
\label{arc} Let us assume that $\mathsf{Re}(\psi(u))\rightarrow\infty$ as
$u\rightarrow\infty$ and that%
\begin{equation}
\label{ImReCond}\frac{\left\vert \mathsf{Im}(\psi(u))\right\vert }%
{\mathsf{Re}(\psi(u))}<A<\infty
\end{equation}
for all $u>0$ and some $A>0.$ Moreover, let $p_{T}$ be (essentially) bounded.
Then, for $0<\Re(z)<1$ it holds that
\[
\mathcal{M}[\mathcal{L}[p_{T}]](z)=\int_{0}^{\infty}u^{z-1}\mathcal{L}%
[p_{T}](u)du=\int_{\ell}w^{z-1}\mathcal{L}[p_{T}](w)dw.
\]

\end{prop}

\begin{rem}
The condition (\ref{ImReCond}) is fulfilled if, for example, the diffusion
part of $L$ is nonzero or if $\psi$ is real and $\psi(u)\to\infty$ as
$u\to\infty.$
\end{rem}

Under the assumptions of Proposition \ref{arc} we may write,%
\[
\mathcal{M}[\mathcal{L}[p_{T}]](z)=\int_{0}^{\infty}\left[  \psi
(\lambda)\right]  ^{z-1}\mathcal{L}[p_{T}](\psi(\lambda))\psi^{\prime}%
(\lambda)d\lambda,
\]
where $\mathcal{L}[p_{T}](\psi(\lambda))=\mathcal{F}[p_{X}](\lambda)$ due to
\eqref{cfx_main}. On other hand, one may straightforwardly derive,
\[
\mathcal{M}[\mathcal{L}[p_{T}]](z)=\mathcal{M}[p_{T}](1-z)\Gamma
(z),\quad0<\mathsf{Re}(z)<1,
\]
i.e.,
\begin{equation}
\mathcal{M}[p_{T}](z)=\frac{\mathcal{M}[\mathcal{L}[p_{T}]](1-z)}{\Gamma
(1-z)}=\frac{\int_{0}^{\infty}\left[  \psi(\lambda)\right]  ^{-z}%
\mathcal{F}[p_{X}](\lambda)\psi^{\prime}(\lambda)d\lambda}{\Gamma(1-z)}%
,\quad0<\mathsf{Re}(z)<1. \label{FpX}%
\end{equation}
In principle, one can now replace the Fourier transform of $p_{X}$ in
\eqref{FpX} by its empirical counterpart based on the data. However, in this
case we need to regularize the estimate of $\mathcal{M}[p_{T}](z)$ to perform
the inverse Mellin transform. To this end consider the approximation
\[
\mathcal{M}[\mathcal{L}[p_{T}]](z)\approx\frac{1}{n}\sum_{k=1}^{n}\int
_{0}^{A_{n}}\left[  \psi(\lambda)\right]  ^{z-1}e^{{\i}X_{k}\lambda}%
\psi^{\prime}(\lambda)d\lambda=:\frac{1}{n}\sum_{k=1}^{n}\Phi_{n}(z,X_{k})
\]
and define in view of (\ref{FpX}),%
\begin{align}
\label{pTn_gen}p_{T,n}(x):=\frac{1}{2\pi n}\sum_{k=1}^{n}\int_{-U_{n}}^{U_{n}%
}\frac{\Phi_{n}(1-\gamma-{\i}v,X_{k})}{\Gamma(1-\gamma-{\i}v)}x^{-\gamma-\i
v}dv,\text{ \ \ for \ }0<\gamma<1\text{\ }%
\end{align}
where $U_{n},A_{n}\rightarrow\infty$ in a suitable way as $n\rightarrow
\infty.$ Note that in many cases the function $\Phi_{n}$ can be found in
closed form. For example, consider the case of a subordinated stable L\'{e}vy
process with $\psi(\lambda)=|\lambda|^{\alpha}.$ It then holds for
$\Re(z)>0,$
\begin{align*}
\Phi_{n}(z,x)  &  =\int_{0}^{A_{n}}\left[  \psi(\lambda)\right]  ^{z-1}e^{{\i
}x\lambda}\psi^{\prime}(\lambda)d\lambda\\
&  =\alpha\int_{0}^{A_{n}}\lambda^{\alpha(z-1)}e^{{\i}x\lambda}\lambda
^{\alpha-1}\,d\lambda\\
&  =\alpha\int_{0}^{A_{n}}\lambda^{\alpha z-1}e^{{\i}x\lambda}d\lambda\\
&  =\frac{A_{n}^{\alpha z}}{z}F_{1}(\alpha z;1+\alpha z;{\i}A_{n}x),
\end{align*}
where $F_{1}$ is Kummer's function. In the next two theorems we prove a
remarkable result showing that the estimate $p_{T,n}(x)$ converges to $p(x)$
at the same rate (up to a logarithmic factor in the polynomial case) as in the
case of the time-changed Brownian motion.

\begin{thm}
\label{polR} Suppose that $\psi$ satisfies the conditions of Proposition
\ref{arc}, and that moreover $\int_{\{|x|>1\}}|x|\nu(dx)<\infty.$ Furthermore
suppose that there is a $1/2<\gamma<1$\ such that $p_{T}\in\mathcal{C}%
(\beta,\gamma,L)$ (cf. Theorem \ref{sep_conv_rates}) for some $\beta>0,$ and
\begin{equation}
\int_{{1}}^{\infty}\frac{1}{\lambda^{2\gamma-1-\epsilon}}\left\vert
\mathcal{F}[p_{X}](\lambda)\right\vert d\lambda<\infty, \label{ae}%
\end{equation}
for some $\epsilon>0.$  Then under the choice
\begin{equation}
A_{n}=n^{\frac{1}{4\left(  1-\gamma\right)  +2\epsilon}} \label{ch1}%
\end{equation}
and%
\begin{equation}
U_{n}=\frac{\epsilon}{\left(  2-2\gamma+\epsilon\right)  \left(  2\beta
+\pi\right)  }\log n-\frac{2\gamma-1}{2\beta+\pi}\log\log n, \label{ch2}%
\end{equation}
we get
\begin{equation}
\sup_{x\geq0}\sqrt{\mathrm{E}\left[  x^{2\gamma}\left\vert p_{T,n}%
(x)-p_{T}(x)\right\vert ^{2}\right]  }\lesssim n^{-\frac{\beta}{2\beta+\pi
}\frac{\epsilon}{2\left(  1-\gamma\right)  +\epsilon}}\log^{\beta\frac
{2\gamma-1}{2\beta+\pi}}n,\quad n\rightarrow\infty. \label{arr_levy}%
\end{equation}
Thus for $\gamma\to1$ or $\epsilon\to0,$ we recover the rates of
Theorem~\ref{sep_conv_rates} up to a logarithmic factor.
\end{thm}

\begin{rem}
\label{rem_four_cond} Since
\begin{align*}
\int_{1}^{\infty}\frac{1}{\lambda^{2\gamma-1-\epsilon}}\left\vert
\mathcal{F}[p_{X}](\lambda)\right\vert d\lambda=\int_{1}^{\infty}\frac
{1}{\lambda^{2\gamma-1-\epsilon}}\left\vert \mathcal{L}[p_{T}](\psi
(\lambda))\right\vert d\lambda,
\end{align*}
the condition \eqref{ae} is, for example, fulfilled for some $\epsilon>0$ if
$\Re[\psi(\lambda)]\gtrsim\lambda$ for $\lambda\to+\infty$ and $p_{T}$ is of
bounded variation with $p_{T}(0)<\infty.$
\end{rem}

In the case $p_{T}\in\mathcal{D}(\beta,\gamma,L),$ we get exactly the same
logarithmic rates as in Theorem~\ref{sep_log_rates}.

\begin{thm}
\label{logR} Suppose that $\psi$ and $\gamma$ are as in Theorem \ref{polR},
and that now $p_{T}\in\mathcal{D}(\beta,\gamma,L)$ (cf. Theorem
\ref{sep_log_rates}) for some $\beta>0.$ Further suppose that (\ref{ae})
holds. Then under the choice
\begin{equation}
A_{n}=n^{\frac{1}{4\left(  1-\gamma\right)  +2\epsilon}} \label{ch1l}%
\end{equation}
(hence the same as in Theorem \ref{polR}) and
\begin{equation}
U_{n}=\frac{\epsilon}{\pi\left(  2-2\gamma+\epsilon\right)  }\log
n-\frac{2\beta+2\gamma-1}{\pi}\log\log n, \label{ch2l}%
\end{equation}
we get
\[
\sup_{x\geq0}\sqrt{\mathrm{E}\left[  x^{2\gamma}\left\vert p_{T,n}%
(x)-p_{T}(x)\right\vert ^{2}\right]  }\lesssim\log^{-\beta}(n),\quad
n\rightarrow\infty.
\]

\end{thm}

\paragraph{Discussion}

The rates in Theorem~\ref{polR} and Theorem~\ref{logR} are optimal in minimax
sense, since they are basically coincides (up to a logarithmic factor) with
the rates in Theorem~\ref{sep_conv_rates} and Theorem~\ref{sep_log_rates},
respectively. As can be seen from the proof of Theorem~\ref{low_bound} and
Remark~\ref{rem_four_cond}, the lower bonds continue to hold true under the
additional assumption \eqref{ae}. Let us also stress that the class
$\mathcal{C}(\beta,\gamma,L)$ is quite large and contains the well known
families of distributions such as Gamma, Beta and Weibull families. It follows
from Theorem~\ref{polR} that for all these families our estimator $p_{n,T}$
converges at a polynomial rate.

\section{Applications}

\subsection{Estimation of the variance-mean mixture models}

The variance-mean mixture of the normal distribution is defined as
\begin{align*}
p(x)=\int_{0}^{\infty}(2\pi\sigma^{2} u)^{-1/2}\exp(-(x-\mu u)^{2}%
/(2\sigma^{2} u))\, g(u) du,
\end{align*}
where $g(u)$ is a mixing density on $\mathbb{R}_{+}.$ The variance-mean
mixture models play an important role in both the theory and the practice of
statistics. In particular, such mixtures appear as limit distributions in
asymptotic theory for dependent random variables and they are useful for
modeling data stemming from heavy-tailed and skewed distributions, see, e.g.
\cite{barndorff1982normal} and \cite{bingham2002semi}. As can be easily seen,
the variance-mean mixture distribution $p$ coincides with the distribution of
the random variable $\sigma W_{T}+\mu T, $ where $T$ is the random variable
with density $g,$ which is independent of $W.$ The class of variance-mean
mixture models is rather large. For example, the class of the normal variance
mixture distributions ($\mu=0$) can be described as follows: $p$ is the
density of a normal variance mixture (equivalently $p$ is the density of
$W_{T}$) if and only if $\mathcal{F}[p](\sqrt{u})$ is a completely monotone
function in $u.$ The problem of statistical inference for variance-mean
mixture models has been already considered in the literature. For example,
Korsholm, \cite{korsholm2000semiparametric} proved the consistency of the
non-parametric maximum likelihood estimator for the parameters $\sigma$ and
$\mu,$ $g$ being treated as an infinite dimensional nuisance parameter. In
Zhang \cite{zhang1990fourier} the problem of estimating the mixing density in
location (mean) mixtures was studied.  To the best of our knowledge, we here
address, for the first time, the problem of non-parametric inference for the
mixing density $g$ in full generality and derive the minimax convergence
rates. In fact, Theorem~\ref{polR} and Theorem~\ref{logR} directly apply not
only to normal variance-mean mixture models, but also to stable variance-mean mixtures.

\subsection{Estimation of time-changed L\'evy models}

Let $L = (L_{t})_{t\geq0} $ be a one-dimensional L\'evy process and let
$\mathcal{T} = (\mathcal{T}(s))_{s\geq0} $ be a non-negative, non-decreasing
stochastic process independent of $X $ with $\mathcal{T}(0)=0 $. A
time-changed L\'evy process $Y = (Y_{s})_{s\geq0} $ is then defined as $Y_{s}
= X_{\mathcal{T}(s)}. $ The process $\mathcal{T} $ is usually referred to as
time change or subordinator. Consider the problem of statistical inference on
the distribution of the time change $\mathcal{T} $ based on the low-frequency
observations of the time-changed L\'evy process $X_{t}=L_{\mathcal{T}(t)}. $
Suppose that $n$ observations of the L\'evy process $L_{t} $ at times
$t_{j}=j\Delta, $ $j=0,\ldots, n, $ are available. If the sequence
$\mathcal{T}(t_{j})-\mathcal{T}(t_{j-1}), $ $j=1,\ldots, n, $ is strictly
stationary with the invariant stationary distribution $\pi, $ then for any
bounded ``test function'' $f,$
\begin{align}
\label{CONVERG}\frac{1}{n}\sum_{j=1}^{n}f\left(  L_{\mathcal{T}(t_{j}%
)}-L_{\mathcal{T}(t_{j-1})}\right)  \to\mathbb{E}_{\pi}[f(L_{\mathcal{T}%
(\Delta)})], \quad n\to\infty,
\end{align}
The limiting expectation in \eqref{CONVERG} is then given by
\begin{align*}
\mathbb{E}_{\pi}[f(L_{\mathcal{T}(\Delta)})]=\int_{0}^{\infty} \mathbb{E}%
[f(L_{s})]\, \pi(ds).
\end{align*}
Taking $f(z) = f_{u}(z)=\exp(\i u^{\top}z), $ $u\in\mathbb{R}^{d}, $ we arrive
at the the following representation for the c.f. of $L_{\mathcal{T}(s)} $:
\begin{align}
\label{EstEq}\mathbb{E}\left[  \exp\left(  \i u L_{\mathcal{T}(\Delta
)}\right)  \right]  =\int_{0}^{\infty} \exp(t\psi(u))\, \pi(dt)=\mathcal{L}%
_{\pi}(\psi(u)),
\end{align}
where $\psi(u):=-t^{-1}\log(\phi_{t}(u))$ with $\phi_{t}(u)=\mathbb{E}\exp(\i
u^{\top}L_{t})$ being the characteristic exponent of the L\'evy process $L $
and $\mathcal{L}_{\pi} $ being the Laplace transform of $\pi. $ Suppose we
want to estimate the invariant measure $\pi$ (or its density) from the
discrete time observations of $L_{\mathcal{T}},$ then we are in the setting of
the generalized statistical Skorohod embedding with the only difference that
the elements of the sample $L_{\mathcal{T}(t_{1})}-L_{\mathcal{T}(t_{0}%
)},\ldots, L_{\mathcal{T}(t_{n})}-L_{\mathcal{T}(t_{n-1})}$ are not
necessarily independent. However, under appropriate mixing properties of the
sequence $\mathcal{T}(t_{j})-\mathcal{T}(t_{j-1}), $ $j=1,\ldots, n, $ one can
easily generalize the results of Section~\ref{seq: gsse} to the case of
dependent data (see, e.g. \cite{belomestny2011statistical} for similar
results). The problem of estimating the parameters of a L\'evy process observed at low frequency was considered in Neumann and Rei{\ss}, \cite{neumann2009nonparametric} and Chen et al, \cite{chen2010nonparametric}.  Let us note that the statistical inference for time-changed L\'evy
processes based on high-frequency observations of $Y$ has been the subject of
many studies, see, e.g. Bull, \cite{bull2013estimating} and Todorov and
Tauchen, \cite{todorov2012realized} and the references therein.

\section{Numerical examples}

Barndorff-Nielsen et al. \cite{barndorff1982normal} consider a class of
variance-mean mixtures of normal distributions which they call generalized
hyperbolic distributions. The univariate and symmetric members of this family
appear as normal scale mixtures whose mixing distribution is the generalized
inverse Gaussian distribution with density
\begin{align}
\label{invGauss}p_{T}(v)=\frac{(\varkappa/\delta)^{\lambda}}{2K_{\lambda
}(\delta\varkappa)} v^{\lambda-1} \exp\left(  -\frac{1}{2}\left(
\varkappa^{2} v+\frac{\delta^{2}}{v}\right)  \right)  , \quad v>0,
\end{align}
for some $\varkappa,$ $\delta\geq0$ and $\lambda>0,$ where $K$ is a modified
Bessel function. The resulting normal scale mixture has probability density
function
\begin{align*}
p_{X}(x)=\frac{\varkappa^{1/2}}{(2\pi)^{1/2}\delta^{\lambda}} K_{\lambda
}(\delta\varkappa) (\delta^{2}+x^{2})^{\frac{1}{2}\left(  \lambda-\frac{1}%
{2}\right)  }K_{\lambda-\frac{1}{2}}\bigl(\varkappa(\delta^{2}+x^{2}%
)^{1/2}\bigr).
\end{align*}
Let us start with a simple example, Gamma density $p_{T}(x)=x\exp(-x),$
$x\geq0,$ which is a special case of \eqref{invGauss} for $\delta=0,$
$\lambda=2$ and $\varkappa=\sqrt{2}.$ We simulate a sample of size $n$ from
the distribution of $X,$ and construct the estimate \eqref{p_Tn_bm} with the
bandwidth $h_{n}$ given (up to a constant not depending on $n$) by \eqref{hn}
and $\gamma=0.8.$ In Figure~\ref{fig: gamma_bm} (left), one can see $50$
estimated densities based on $50$ independent samples from $W_{T}$ of size
$n=1000,$ together with $p_{T}$ in red. Next we estimate the distribution of
the loss $\sup_{x\in[0,10]}\bigl\{ |p_{T,n}(x)-p_{T}(x)|\bigr\}$ based on
$100$ independent repetitions of the estimation procedure. The corresponding
box plots for different $n$ are shown in Figure~\ref{fig: gamma_bm} (right).
\begin{figure}[t]
\centering
\includegraphics[width=0.45\linewidth]{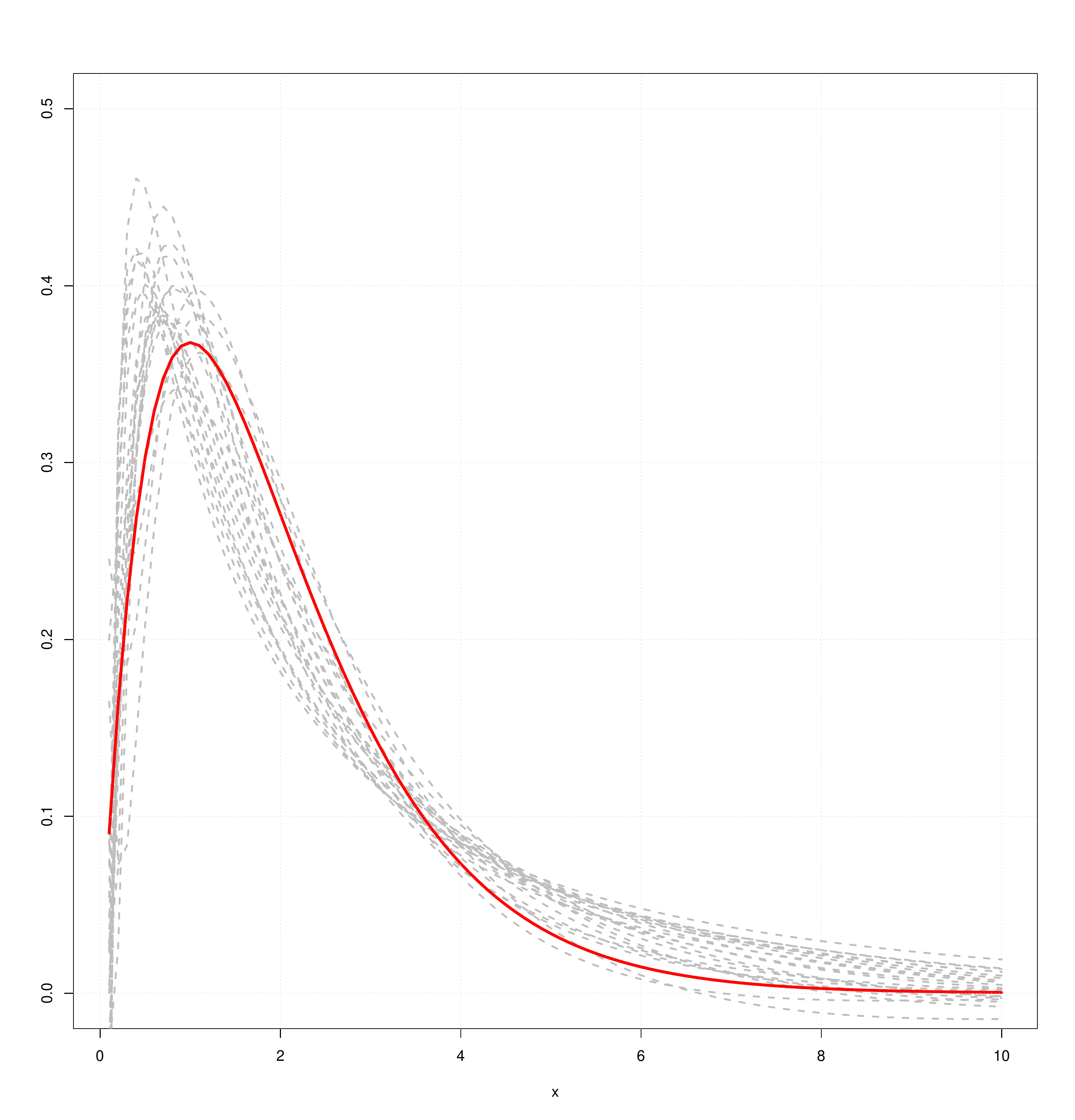}
~\includegraphics[width=0.45\linewidth]{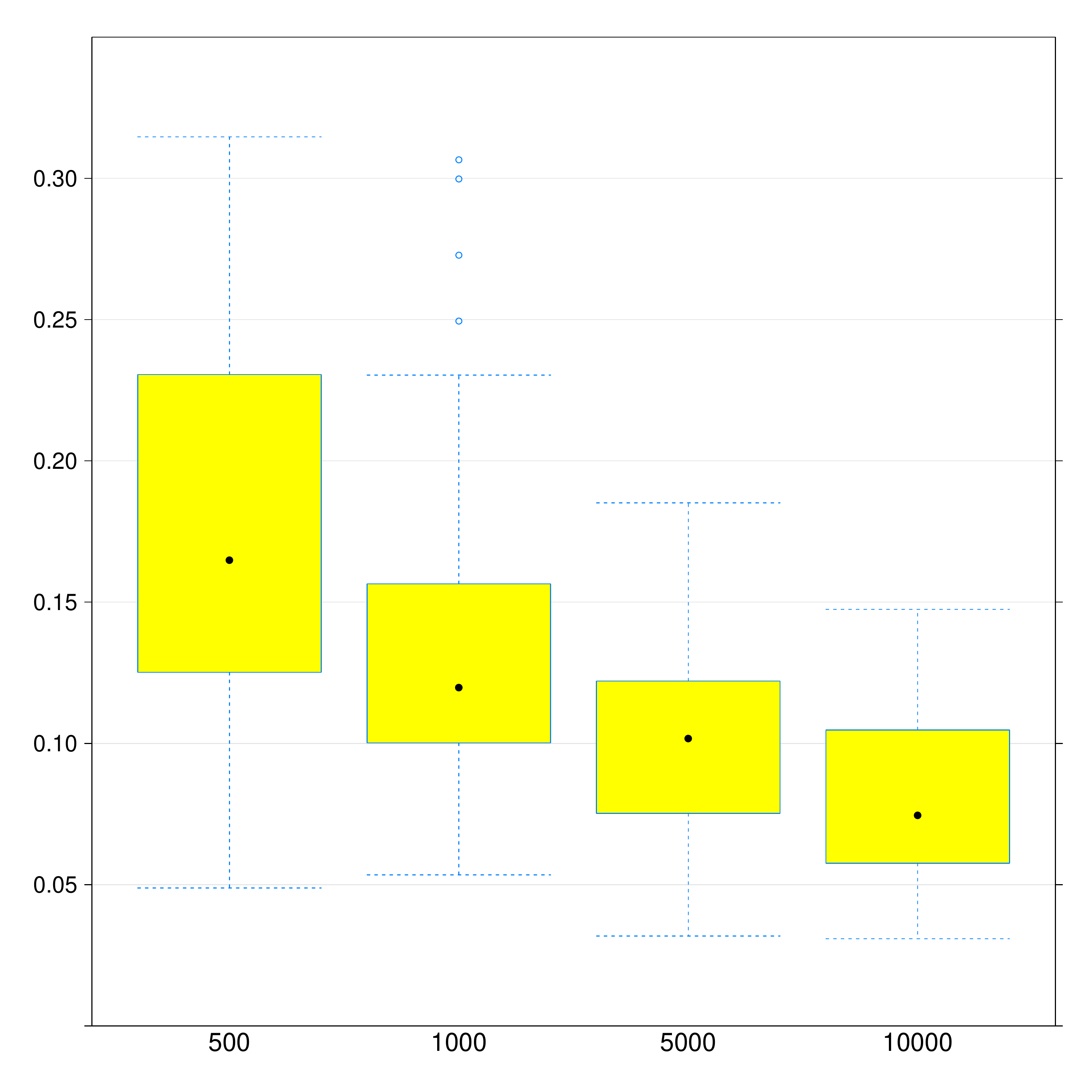}\caption{Left: the
Gamma density (red) and its $50$ estimates (grey) for the sample size
$n=1000$. Right: the box plots of the loss $\sup_{x\in[0,10]}\bigl\{ |p_{T,n}%
(x)-p_{T}(x)|\bigr\}$ for different sample sizes.}%
\label{fig: gamma_bm}%
\end{figure}

Let us now turn to a more interesting example of variance-mean mixtures. We
take $X=T+W_{T}$ and choose $T$ to follow a Gamma distribution with the
density $p_{T}(x)=x\exp(-x),$ $x\geq0.$ The estimate \eqref{pTn_gen} is
constructed as follows. First note that $\psi(\lambda)=-\i\lambda+\lambda
^{2}/2.$ In order to numerically compute the function $\Phi_{n}(1-z,X_{k})$
for $z=\gamma+\i v$ with $\gamma<1,$ we use the decomposition
\begin{align}
\label{phin_approx}\frac{1}{n}\sum_{k=1}^{n}\Phi_{n}(1-z,X_{k}) & =\int
_{0}^{A_{n}}\left[  \psi(\lambda)\right]  ^{-z}[\phi_{n}(\lambda
)-e^{-m_{n}\psi(\lambda)}]\psi^{\prime}(\lambda)\,d\lambda\\
& +m_{n}^{z-1}\Gamma(1-z)+O\bigl(m_{n}^{-(1-\gamma)}\exp(-m_{n}A_{n}%
^{2}/2)\bigr),\nonumber
\end{align}
where $\phi_{n}(\lambda)=\frac{1}{n} \sum_{k=1}^{n} e^{\i\lambda X_{k}}$ is
the empirical characteristic function and $m_{n}=\frac{1}{n}\sum_{k=1}^{n}
X_{k}\to2.$ This decomposition follows from a Cauchy argument similar as in
the proof of Proposition~\ref{arc} and is quite useful to reduce the cost of
computing the integral in \eqref{phin_approx}, since the integral on the
r.h.s. of \eqref{phin_approx} is much easier to compute due to the asymptotic
relation $\phi_{n}(\lambda)-e^{-m_{n}\psi(\lambda)}=O(\lambda^{2}),$
$\lambda\to0.$ Next we take $\gamma=0.7,$ $A_{n}$ and $h_{n}$ as in
Theorem~\ref{polR} with $\varepsilon=0.5$ and $\beta=\pi/2$ (see
Example~\ref{exmp_gamma}). Figure~\ref{fig: gamma_bmd} shows the performance
of the estimate defined in \eqref{pTn_gen}: on the left-hand side $20$
independent realizations of the estimate $p_{T,n}$ for $n=1000$ are shown
together with the true density $p_{T}.$ The box plots of the loss $\sup
_{x\in[0,10]}\bigl\{ |p_{T,n}(x)-p_{T}(x)|\bigr\}$ based on $100$ runs of the
algorithm are depicted on the right-hand side of Figure~\ref{fig: gamma_bmd}.
By comparing the right-hand sides of Figure~\ref{fig: gamma_bm} and
Figure~\ref{fig: gamma_bmd}, we observe that the performances of the estimates
\eqref{pTn_gen} and \eqref{p_Tn_bm} are similar, although the estimate
\eqref{p_Tn_bm} seem to have higher variance. This supports the claim of
Theorem~\ref{polR} about the same convergence rates in statistical Skorohod
embedding and generalized statistical Skorohod embedding problems, given that
$p_{T}\in\mathcal{C}(\beta,\gamma,L).$ \begin{figure}[tbh]
\centering
\includegraphics[width=0.45\linewidth]{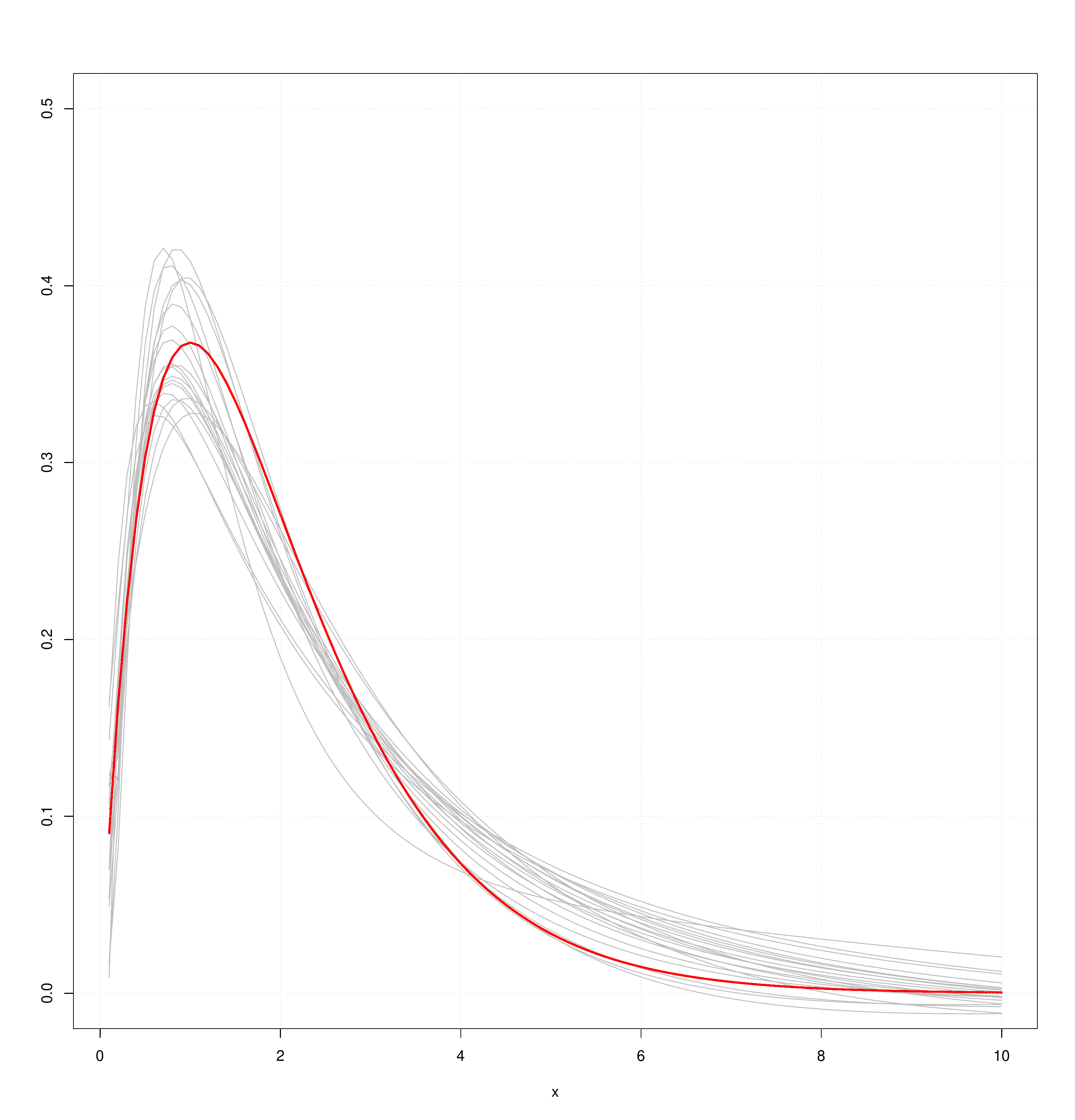}
~\includegraphics[width=0.45\linewidth]{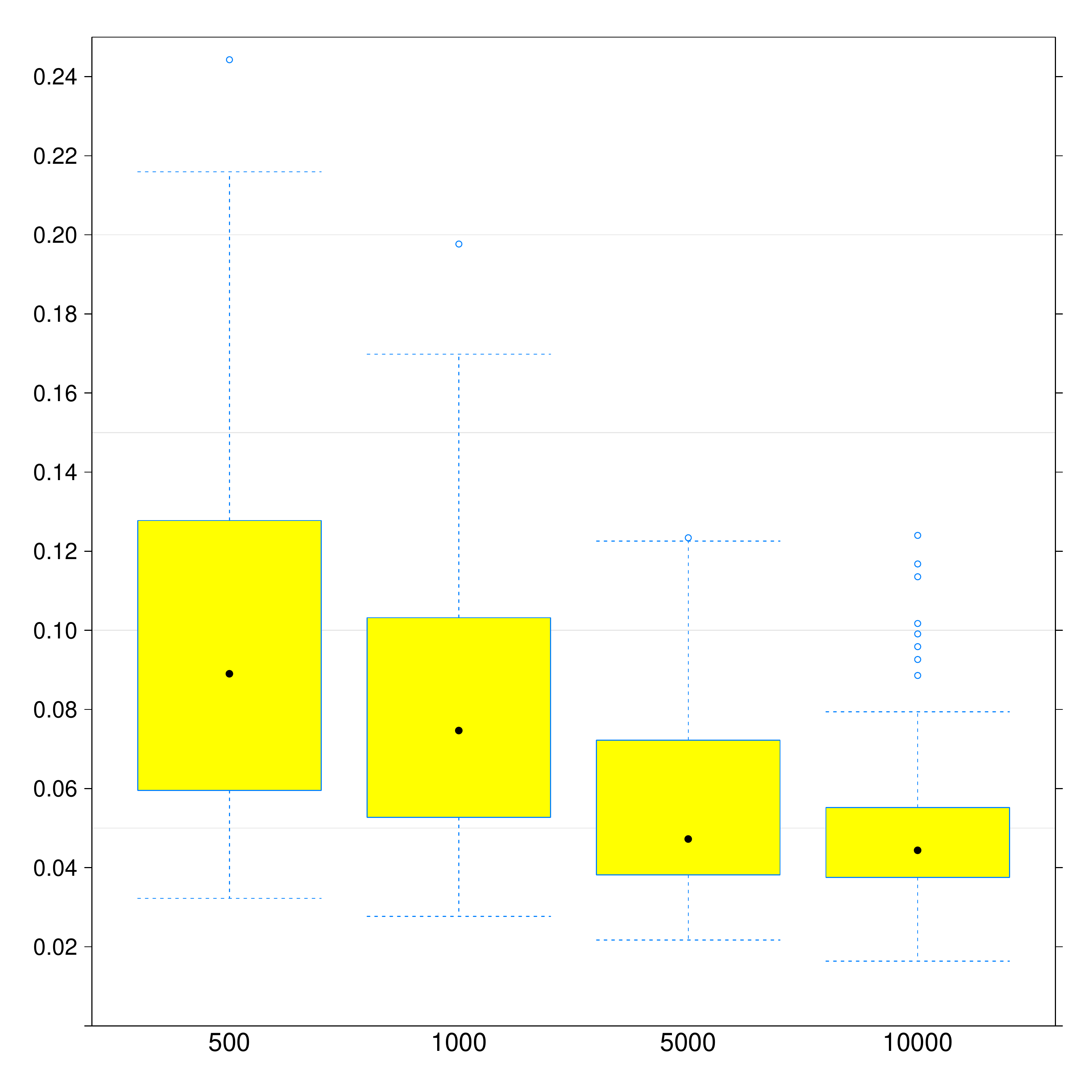}\caption{Left:
the Gamma density (red) and its $20$ estimates (grey) for the sample size
$n=5000$. Right: the box plots of the loss $\sup_{x\in[0,10]}\bigl\{ |p_{T,n}%
(x)-p_{T}(x)|\bigr\}$ for different sample sizes.}%
\label{fig: gamma_bmd}%
\end{figure}

\section{Proofs}

\subsection{Proof of Theorem~\ref{sep_conv_rates}}

First let us estimate the bias of $p_{T,n}.$ We have
\begin{align*}
\mathbb{E}[p_{T,n}(x)]  &  =\frac{1}{\sqrt{\pi}}\int_{-\infty}^{\infty
}x^{-\gamma-\i v}K(vh_{n})\frac{\mathcal{M}[p_{|X|}](2(\gamma+\i
v)-1)}{2^{\gamma+\i v}\Gamma(\gamma-1/2+\i v)}\,dv\\
&  =\frac{1}{2\pi}\int_{-1/h_{n}}^{1/h_{n}}x^{-\gamma-\i v}\mathcal{M}%
[p_{T}](\gamma+\i v)\,dv.
\end{align*}
Hence
\[
p_{T}(x)-\mathbb{E}[p_{T,n}(x)]=\frac{1}{2\pi}\int_{\{\left\vert v\right\vert
\geq1/h_{n}\}}\mathcal{M}[p_{T}](\gamma+\i v)x^{-\gamma-\i v}dv
\]
and we then have the estimate,
\begin{align}
\sup_{x\geq0}\bigl\{x^{\gamma}|\mathbb{E}[p_{T,n}(x)]-p_{T}(x)|\bigr\}  &
\leq\frac{1}{2\pi}\int_{\{\left\vert v\right\vert \geq1/h_{n}\}}\left\vert
\mathcal{M}[p_{T}](\gamma+\i v)\right\vert dv\nonumber\\
&  \leq\frac{e^{-\beta/h_{n}}}{2\pi}\int_{\{\left\vert v\right\vert
\geq1/h_{n}\} }e^{-\beta\left\vert v\right\vert }\left\vert \mathcal{M}%
[p_{T}](\gamma+\i v)\right\vert e^{\beta\left\vert v\right\vert }
dv\nonumber\\
&  \leq L\,\frac{e^{-\beta/h_{n}}}{2\pi}. \label{bias}%
\end{align}
As to the variance, by the simple inequality $\operatorname{Var}\left(  \int
f_{t}dt\right)  \leq\left(  \int\sqrt{\operatorname{Var}[f_{t}]} dt\right)
^{2},$ which holds for any random function $f_{t}$ with $\int\mathbb{E}%
[f^{2}_{t}] dt<\infty,$ we get
\begin{align}
\operatorname{Var}[x^{\gamma}p_{T,n}(x)]  &  =\operatorname{Var}\left[
\frac{1}{\sqrt{\pi}}\int_{-\infty}^{\infty}x^{-\i v}K(vh_{n})\frac
{\mathcal{M}_{n}[p_{|X|}](2(\gamma+\i v)-1)}{2^{\gamma+iv}\Gamma(\gamma-1/2+\i
v)}\,dv\right] \nonumber\\
&  \leq\frac{1}{\pi2^{2\gamma}}\left[  \int_{-1/h_{n}}^{1/h_{n}}\frac
{\sqrt{\operatorname{Var}\left(  \mathcal{M}_{n}[p_{|X|}](2(\gamma+\i
v)-1)\right)  }}{\left\vert \Gamma(\gamma-1/2+\i v)\right\vert }dv\right]
^{2}\,\nonumber\\
&  \leq\frac{1}{2n\pi}\left[  \int_{-1/h_{n}}^{1/h_{n}}\frac{\sqrt
{\operatorname{Var}\bigl(|X|^{2(\gamma+\i v-1)}\bigr)}}{\left\vert
\Gamma(\gamma-1/2+\i v)\right\vert }dv\right]  ^{2}\,\nonumber\\
&  \leq\frac{1}{2n\pi}\left[  \int_{-1/h_{n}}^{1/h_{n}}\frac{\sqrt
{\mathbb{E}\left[  |W_{T}|^{4(\gamma-1)}\right]  }}{\left\vert \Gamma
(\gamma-1/2+\i v)\right\vert }\, dv\right]  ^{2}. \label{vr}%
\end{align}
Note that
\begin{align*}
\mathbb{E}\left[  |W_{T}|^{4(\gamma-1)}\right]   &  =\int_{0}^{\infty
}\mathbb{E}\left[  |W_{t}|^{4(\gamma-1)}\right]  p_{T}(t)\, dt\\
&  =\mathbb{E}\left[  |W_{1}|^{4(\gamma-1)}\right]  \int_{0}^{\infty
}t^{2(\gamma-1)}p_{T}(t)\,dt\\
&  =:C_{2}(\gamma)<\infty,
\end{align*}
due to (\ref{sg}).  We obtain from (\ref{vr}) due to
Corollary~\ref{cor_integ_gamma} (see Appendix) and by taking into account
(\ref{sg}),%
\[
\operatorname{Var}[x^{\gamma}p_{T,n}(x)]\leq\frac{C_{2}(\gamma)}{2n\pi}%
C_{3}h_{n}^{2\left(  \gamma-1\right)  }e^{\pi/h_{n}}=\frac{C_{3}(\gamma)}%
{n}h_{n}^{2\left(  \gamma-1\right)  }e^{\pi/h_{n}}.
\]
and so (\ref{ms}) follows with $C_{\gamma,L}=\max(C_{3}(\gamma),\frac{L^{2}%
}{4\pi^{2}}).$ Finally, by plugging (\ref{hn}) into (\ref{ms}) we get
(\ref{rates_bm_pol}) and the proof is finished.

\subsection{Proof of Theorem ~\ref{sep_log_rates}}

The proof is analog to the one of Theorem~\ref{sep_conv_rates} , the only
difference is the bias estimate (\ref{bias}) that now becomes
\[
\sup_{x\geq0}\bigl\{x^{\gamma}|\mathbb{E}[p_{T,n}(x)]-p_{T}(x)|\bigr\}\leq
\frac{L}{2\pi}h_{n}^{\beta},
\]
which gives (\ref{ms1}) with a constant $D_{\gamma,L}=$ $\max(C_{3}%
(\gamma),\frac{L^{2}}{4\pi^{2}})$ again. Next with the choice (\ref{hn1}) we
obtain from (\ref{ms1}) the logarithmic rate (\ref{arr1}).

\subsection{Proof of Theorem~\ref{low_bound}}

Our construction relies on the following basic result (see
\cite{tsybakov2009introduction} for the proof).

\begin{thm}
\label{ThmLowerBound} Suppose that for some $\varepsilon>0$ and $n\in
\mathbb{N}$ there are two densities $p_{0,n},p_{1,n}\in\mathcal{G}$ such that
\[
d(p_{0,n},p_{1,n})> 2\varepsilon v_{n}.
\]
If the observations in model $n$ follow the product law $\mathsf{P}%
_{p,n}=\mathsf{P}_{p}^{\otimes n}$ under the density $p\in\mathcal{G}$ and
\[
\chi^{2}(p_{1,n}\,|\,p_{0,n})\le n^{-1}\log(1+(2-4\delta)^{2})
\]
holds for some $\delta\in(0,1/2)$, then the following lower bound holds for
all density estimators $\hat p_{n}$ based on observations from model $n$:
\[
\inf_{\hat p_{n}}\sup_{p\in\mathcal{G}}\P ^{\otimes n}_{p}\big(d(\hat
p_{n},p)\ge\varepsilon v_{n}\big)\ge\delta.
\]
If the above holds for fixed $\varepsilon,\delta>0$ and all $n\in\mathbb{N}$,
then the optimal rate of convergence in a minimax sense over $\mathcal{G}$ is
not faster than $v_{n}$.
\end{thm}

\subsubsection{Proof of a lower bound for the class $\mathcal{C}(\beta
,\gamma,L)$}

Let us start with the construction of the densities $p_{0,n}$ and $p_{1,n}.$
Define for any $\nu>1$ and $M>0$ two auxiliary functions
\begin{align*}
q(x)=\frac{\nu\sin(\pi/\nu)}{\pi}\frac{1}{1+x^{\nu}},\quad x\geq0
\end{align*}
and
\begin{align*}
\rho_{M}(x)=\frac{1}{\sqrt{2\pi}}e^{-\frac{\log^{2}(x)}{2}}\frac{\sin
(M\log(x))}{x}, \quad x\geq0.
\end{align*}
The properties of the functions $q$ and $\rho_{M}$ are collected in the
following lemma.

\begin{lem}
\label{l1_proof_low} The function $q$ is a probability density on
$\mathbb{R}_{+}$ with the Mellin transform
\begin{align*}
\mathcal{M}[q](z)=\frac{\sin(\pi/\nu)}{\sin(\pi z/\nu)} ,\quad\mathsf{Re}%
[z]>0.
\end{align*}
The Mellin transform of the function $\rho_{M}$ is given by
\begin{align}
\label{mellin_rho_pol}\mathcal{M}[\rho_{M}](u+\i v)=\frac{1}{2}\left[
e^{(u-1+\i(v+M))^{2}/2}-e^{(u-1+\i(v-M))^{2}/2}\right]  .
\end{align}
Hence
\begin{align*}
\int_{0}^{\infty}\rho_{M}(x) dx=\mathcal{M}[\rho_{M}](1)=0.
\end{align*}

\end{lem}

Set now for any $M>0$
\begin{align*}
q_{0,M}(x):=q(x), \quad q_{1,M}(x):=q(x)+(q\vee\rho_{M}) (x),
\end{align*}
where $f\vee g$ stands for the multiplicative convolution of two functions $f$
and $g$ on $\mathbb{R}_{+}$ defined as
\begin{align*}
(f\vee g)(x):=\int_{0}^{\infty}\frac{f(t) g(x/t)}{t} dt, \quad x\geq0.
\end{align*}
The following lemma describes some properties of $q_{0,M}$ and $q_{1,M}.$

\begin{lem}
\label{l2_proof_low} For any $M>0$ the function $q_{1,M}$ is a probability
density satisfying
\begin{align*}
\|q_{0,M}-q_{1,M}\|_{\infty}=\sup_{x\in\mathbb{R}_{+}}|q_{0,M}(x)-q_{1,M}%
(x)|\gtrsim\exp(-M\pi/\nu) , \quad M\to\infty.
\end{align*}
Moreover, $q_{0,M}$ and $q_{1,M}$ are in $\mathcal{C}(\beta,\gamma,L)$ for all
$0<\beta<\pi/\nu$ and $\gamma>0$ with $L$ depending on $\gamma.$
\end{lem}

\begin{proof}
First note that
\begin{eqnarray*}
\int_0^\infty q_{1,M}(x) dx=1+\int_0^\infty (q\vee \rho_M) (x)=1+\mathcal{M}[q](1)\mathcal{M}[\rho_M](1) = 1.
\end{eqnarray*}
Furthermore, due to the Parseval identity
\begin{eqnarray*}
(q\vee \rho_M) (y)&=&\int_{0}^{\infty}\frac{1}{\sqrt{2\pi}}e^{-\frac{\log^{2}(x)}{2}}\frac{\sin(M\log(x))}{x^{2}}\frac{1}{1+(y/x)^{\nu}}dx	
\\
&=&\int_{-\infty}^{\infty}\frac{1}{\sqrt{2\pi}}e^{-\frac{v^{2}}{2}}\sin(Mv)\frac{e^{-v}}{1+e^{-\nu(v-y_{l})}}dv
\\
&=&e^{-\log(y)}\int_{-\infty}^{\infty}\frac{1}{\sqrt{2\pi}}e^{-\frac{v^{2}}{2}}\sin(Mv)\frac{e^{(y_{l}-v)}}{1+e^{\nu (\log(y)-v)}}dv
\\
&=&	\frac{e^{-\log(y)}}{2\pi}\int_{-\infty}^{\infty}\frac{1}{\sqrt{2\pi}}e^{-\frac{v^{2}}{2}}\sin(Mv)\frac{e^{(\log(y)-v)}}{1+e^{\nu (\log(y)-v)}}dv
\\
&=&	\frac{e^{-\log(y)}}{2\pi}\int_{-\infty}^{\infty}e^{-\i u\log(y)}\left[\frac{H(u+M)-H(u-M)}{2}\right]\mathcal{F}[R](u)du,
\end{eqnarray*}
where \(R(x)=\frac{e^{x}}{1+e^{\nu x}}\) and  \(H(x)=e^{-x^{2}/2}\).
Note that
\begin{eqnarray*}
\mathcal{F}[R](u)=\int_{-\infty}^{\infty}\frac{e^{x+\i ux}}{1+e^{\nu x}}\,dx=\frac{1}{\nu}\int_{-\infty}^{\infty}\frac{e^{v/\nu+\i uv/\nu}}{1+e^{v}}\,dx=\frac{1}{\nu}\,\Gamma\left(\frac{1+\i u}{\nu}\right)\Gamma\left(1-\frac{1+\i u}{\nu}\right).
\end{eqnarray*}
Hence due to \eqref{gamma_asymp}
\[
\sup_{y\in \mathbb{R}_+}|q_{0,M}(y)-q_{1,M}(y)|=\sup_{y\in \mathbb{R}_+}|(q\vee \rho_M) (y)|\gtrsim \exp(-M\pi/\nu), \quad M\to \infty.
\]
The second statement of the lemma follows from Lemma~\ref{l1_proof_low} and the fact that \(\mathcal{M}[q\vee \rho_M]=\mathcal{M}[q]\mathcal{M}[\rho_M].\)
\end{proof}	
Let $T_{0,M}$ and $T_{1,M}$ be two random variables with densities $q_{0,M}$
and $q_{1,M},$ respectively. Then the density of the r.v. $|W_{T_{i,M}}|,$
$i=0,1,$ is given by
\begin{align*}
p_{i,M}(x)  &  :=\frac{2}{\sqrt{2\pi}}\int_{0}^{\infty}\lambda^{-1/2}%
e^{-\frac{x^{2}}{2\lambda}}q_{i,M}(\lambda)\,d\lambda\quad i=0,1.
\end{align*}
For the Mellin transform of $p_{i,M}$ we get
\begin{align}
\label{p_mellin_pol}\mathcal{M}[p_{i,M}](z)  &  =\mathbb{E}\bigl[|W_{1}%
|^{z-1}\bigr]\mathbb{E}\bigl[T_{i,M}^{(z-1)/2}\bigr]\nonumber\\
&  =\mathbb{E}\bigl[|W_{1}|^{z-1}\bigr]\mathcal{M}[q_{i,M}%
]((z+1)/2)\nonumber\\
&  =\frac{2^{z/2}}{\sqrt{2\pi}}\Gamma(z/2)\mathcal{M}[q_{i,M}]((z+1)/2), \quad
i=0,1.
\end{align}

\begin{lem}
\label{l3_proof_low} The $\chi^{2}$-distance between the densities $p_{0,M}$
and $p_{1,M}$ fulfills
\begin{align*}
\chi^{2}(p_{1,M}|p_{0,M})=\int\frac{(p_{1,M}(x)-p_{0,M}(x))^{2}}{p_{0,M}%
(x)}dx\lesssim e^{-M\pi(1+2/\nu)}, \quad M\to\infty.
\end{align*}

\end{lem}

\begin{proof}
First note that \(p_{0,M}(x)>0\) on \([0,\infty).\)
Since
\begin{eqnarray*}
p_{0,M}(x)&=&\frac{2}{\sqrt{2\pi}}\frac{\nu\sin(\pi/\nu)}{\pi}\int_{0}^{\infty}\lambda^{-1/2}e^{-\frac{x^{2}}{2\lambda}}\frac{1}{1+\lambda^{\nu}}d\lambda\\/y=1/\lambda/&=&\frac{2}{\sqrt{2\pi}}\frac{\nu\sin(\pi/\nu)}{\pi}\int_{0}^{\infty}y^{1/2}e^{-y\frac{x^{2}}{2}}\frac{1}{y^{2}(1+y^{-\nu})}dy\\&=&\frac{2}{\sqrt{2\pi}}\frac{\nu\sin(\pi/\nu)}{\pi}\int_{0}^{\infty}e^{-y\frac{x^{2}}{2}}\frac{y^{\nu-1/2-1}}{(1+y^{\nu})}dy\\&\asymp&\frac{2}{\sqrt{2\pi}}\frac{\nu\sin(\pi/\nu)}{\pi}\Gamma(\nu-1/2)x^{-2\nu+1}, \quad x\to \infty,
\end{eqnarray*}
we have \(p_{0,M}(x)\gtrsim x^{-2\nu+1},\) \(x\to \infty.\) Furthermore, due to \eqref{p_mellin_pol} and the Parseval identity
\begin{multline}
\label{parseval}
\int_{0}^{\infty}x^{2\nu-1}\left|p_{0,M}(x)-p_{1,M}(x)\right|^{2}dx=
\\
\frac{2^{-4+2\nu}}{\pi}\int_{\gamma-i\infty}^{\gamma+i\infty}\mathcal{M}[q\vee \rho_M]\left(\frac{z+1}{2}\right)
\Gamma\left(\frac{z}{2}\right)\mathcal{M}[q\vee \rho_M]\left(\frac{2\nu-z+1}{2}\right)\Gamma\left(\frac{2\nu-z}{2}\right)dz,
\end{multline}
where \(\mathcal{M}[q\vee \rho_M](z)=\mathcal{M}[q](z)\mathcal{M}[\rho_M](z).\) Due to \eqref{mellin_rho_pol}
\begin{eqnarray}
\label{mellin_rho_bound}
|\mathcal{M}[\rho_{M}](u+\i v)|\leq e^{\frac{(u-1)^{2}}{2}}\frac{\phi(v+M)+\phi(v-M)}{2}
\end{eqnarray}
with \(\phi(v)=e^{-\frac{ v^{2}}{2}}.\) Combining \eqref{gamma_asymp} (Appendix), \eqref{parseval}  and  \eqref{mellin_rho_bound}, we derive
\begin{eqnarray*}
\chi^{2}(p_{1,M}|p_{0,M})&=&\int\frac{(p_{1,M}(x)-p_{0,M}(x))^{2}}{p_{0,M}(x)}dx
\\
&\lesssim&\int_{0}^{\infty}(p_{1,M}(x)-p_{0,M}(x))^{2}dx+\int_{0}^{\infty}x^{2\nu-1}(p_{1,M}(x)-p_{0,M}(x))^{2}dx\\&\lesssim&\int_{-\infty}^{\infty}|v|^{\nu-1}e^{-|v|\pi/2-|v|\pi/\nu}\left(\phi(v/2+M)+\phi(v/2-M)\right)^{2}dv
\\
&\lesssim & M^{\nu-1}e^{-M\pi(1+2/\nu)}, \quad M\to \infty.
\end{eqnarray*}
\end{proof}
Fix some $\kappa\in(0,1/2).$ Due to Lemma~\ref{l3_proof_low}, the inequality
\[
n\chi^{2}(p_{1,M}|p_{0,M})\leq\kappa
\]
holds for $M$ large enough, provided
\begin{align*}
M=\frac{1+\epsilon}{\pi(1+2/\nu)}(\log(n)+(\nu-1)\log\log(n))
\end{align*}
for arbitrary small $\epsilon>0.$ Hence Lemma~\ref{l2_proof_low} and
Theorem~\ref{ThmLowerBound} imply
\[
\inf_{\hat p_{n}}\sup_{p\in\mathcal{C}(\beta,\gamma,L)}\mathsf{P}%
_{p,n}\big(\|\hat p_{n}-p\|_{\infty}\ge c v_{n}\big)\ge\delta.
\]
for any $\beta<\pi/\nu<\pi,$ any $\gamma>0,$ some constants $c>0,$ $\delta>0$
and $v_{n}=n^{-\beta/(\pi+2\beta)}\log^{-\frac{\pi-\beta}{\pi+2\beta}}(n).$

\subsubsection{Proof of a lower bound for the class $\mathcal{D}(\beta
,\gamma,L)$}

Define for any $\nu>1,$ $\alpha>0$ and $M>0,$
\begin{align*}
q(x)=\left[  2\Gamma(\nu)\right]  ^{-1}\times%
\begin{cases}
\log^{\nu-1}(1/x), & 0\leq x\leq1,\\
x^{-2}\log^{\nu-1}(x), & x>1
\end{cases}
\end{align*}
and
\begin{align*}
\rho_{M}(x)=\frac{1}{\sqrt{2\pi}}e^{-\frac{\log^{2}(x)}{2}}\frac{\sin
(M\log(x))}{x\log(x)}, \quad x\geq0.
\end{align*}
The properties of the functions $q$ and $\rho_{M}$ can be found in the next lemma.

\begin{lem}
\label{l1_proof_low_log} The function $q$ is a probability density on
$\mathbb{R}_{+}$ with the Mellin transform
\begin{align*}
\mathcal{M}[q](z)=\frac{1}{2}\left[  z^{-\nu}+(2-z)^{-\nu}\right]  ,\quad0<
\mathsf{Re}[z]< 2.
\end{align*}
The Mellin transform of the function $\rho_{M}$ is given by
\begin{align}
\label{mellin_rho}\mathcal{M}[\rho_{M}](u+\i v)=e^{\frac{(u-1)^{2}}{2}}%
\frac{G(u,v+M)-G(u,v-M)}{2},
\end{align}
where $G(u,v)=\int_{-\infty}^{v}e^{-\frac{x^{2}}{2}+\i x(u-1)}dx.$ Hence
\begin{align*}
\zeta_{M}:=\int_{0}^{\infty}\rho_{M}(x) dx=\mathcal{M}[\rho_{M}](1)=\int
_{-M}^{M}e^{-\frac{x^{2}}{2}}dx.
\end{align*}

\end{lem}

Set now for any $M>0$
\begin{align*}
q_{0,M}(x):=q(x), \quad q_{1,M}(x):=(1-\zeta_{M})q(x)+(q\vee\rho_{M}) (x),
\end{align*}
where $f\vee g$ stands for the multiplicative convolution of two functions $f$
and $g$ on $\mathbb{R}_{+}$ defined via
\begin{align*}
(f\vee g)(x):=\int_{0}^{\infty}\frac{f(t) g(x/t)}{t} dt.
\end{align*}

\begin{lem}
\label{l2_proof_low_log} For any $M>0,$ the function $q_{1,M}$ is a
probability density satisfying
\begin{align*}
\sup_{x\in(1-\delta,1+\delta)}|q_{0,M}(x)-q_{1,M}(x)|\asymp|\cos(\pi
\nu/2)|M^{-\nu+1}, \quad M\to\infty,
\end{align*}
where $\delta>0$ is a fixed number. Moreover, $q_{0,M}$ and $q_{1,M}$ are in
$\mathcal{D}(\beta,\gamma,L)$ for all $\beta<\nu-1$ and $\gamma\in(0,2).$
\end{lem}

\begin{proof}
First note that
\begin{eqnarray*}
\int_0^\infty q_{1,M}(x) dx=1+\int_0^\infty (q\vee \rho_M) (x) - \zeta_M= 1+\mathcal{M}[\rho_{M}](1)\times\mathcal{M}[q](1)-\zeta_M=1.
\end{eqnarray*}
Furthermore, \((q\vee \rho_M) (y)=\left[2\Gamma(\nu)\right]^{-1}[I_1(y)+I_2(y)]\) with
\begin{eqnarray*}
I_{1}(y)&=&\int_{y}^{\infty}e^{-\frac{\log^{2}(x)}{2\alpha}}x^{-2}\frac{\sin(M\log(x))}{\log(x)}\log^{\nu-1}(x/y)dx
\\
&=&\int_{\log(y)}^{\infty}e^{-\frac{z^{2}}{2\alpha}-z}\frac{\sin(Mz)}{z}(z-\log(y))^{\nu-1}dz
\end{eqnarray*}
and
\begin{eqnarray*}
I_{2}(y)&=&\int_{0}^{y}e^{-\frac{\log^{2}(x)}{2\alpha}}y^{-2}\frac{\sin(M\log(x))}{\log(x)}\log^{\nu-1}(y/x)dx
\\
&=&\int_{-\infty}^{\log(y)}e^{-\frac{z^{2}}{2\alpha}+z}y^{-2}\frac{\sin(Mz)}{z}(\log(y)-z)^{\nu-1}dz.
\end{eqnarray*}
By taking \(y=\exp(A),\) we get for \(I_1(y)\)
\begin{eqnarray*}
I_1(y)&=&\int_{0}^{\infty}e^{-\frac{(z+A)^{2}}{2\alpha}-(z+A)}\frac{\sin(M(z+A))}{z+A}z^{\nu-1}dz
\\
&=&\cos(AM)\int_{0}^{\infty}\frac{e^{-\frac{(z+A)^{2}}{2\alpha}-(z+A)}}{z+A}\sin(Mz)z^{\nu-1}dz
\\
&& + \sin(AM)\int_{0}^{\infty}\frac{e^{-\frac{(z+A)^{2}}{2\alpha}-(z+A)}}{z+A}\cos(Mz)z^{\nu-1}dz.
\end{eqnarray*}
The well known Erd\'elyi lemma implies
\begin{eqnarray*}
\int_{0}^{\infty}\frac{e^{-\frac{(z+A)^{2}}{2\alpha}-(z+A)}}{z+A}\sin(Mz)z^{\nu-1}dz\asymp\frac{e^{-\frac{A{}^{2}}{2\alpha}-A}}{A}\Gamma(\nu)\sin(\pi\nu/2)M^{-\nu}, \quad M\to \infty
\end{eqnarray*}
and
\begin{eqnarray*}
\int_{0}^{\infty}\frac{e^{-\frac{(z+A)^{2}}{2\alpha}-(z+A)}}{z+A}\cos(Mz)z^{\nu-1}dz\asymp\frac{e^{-\frac{A{}^{2}}{2\alpha}-A}}{A}\Gamma(\nu)\cos(\pi\nu/2)M^{-\nu}, \quad M\to \infty.
\end{eqnarray*}
Hence
\begin{eqnarray}
\label{I_1(e^A)}
I_1(e^A)\asymp\frac{e^{-\frac{A{}^{2}}{2\alpha}-A}}{A}\Gamma(\nu)\sin(AM+\pi\nu/2)M^{-\nu}, M\to \infty.
\end{eqnarray}
Analogously
\begin{eqnarray*}
I_2(e^A)&=&e^{-2A}\int_{-\infty}^{A}e^{-\frac{z^{2}}{2\alpha}+z}\frac{\sin(Mz)}{z}(A-z)^{\nu-1}dz
\\
&=& e^{-2A}\int_{0}^{\infty}e^{-\frac{(A-z)^{2}}{2\alpha}+A-z}\frac{\sin(M(A-z))}{A-z}z{}^{\nu-1}dz
\\
&=& e^{-2A}\sin(AM)\int_{0}^{\infty}e^{-\frac{(A-z)^{2}}{2\alpha}+A-z}\frac{\cos(Mz)}{A-z}z{}^{\nu-1}dz
\\
&& -e^{-2A}\cos(AM)\int_{0}^{\infty}e^{-\frac{(A-z)^{2}}{2\alpha}+A-z}\frac{\sin(Mz)}{A-z}z{}^{\nu-1}dz
\\
&\asymp& \frac{e^{-\frac{A^{2}}{2\alpha}-A}}{A}\Gamma(\nu)\sin(AM-\pi\nu/2)M^{-\nu}.
\end{eqnarray*}
Combining the previous estimates, we arrive at
\begin{eqnarray*}
I_2(e^A)+I_1(e^A)=2\frac{e^{-\frac{A^{2}}{2\alpha}-A}}{A}\Gamma(\nu)\sin(AM)\cos(\pi\nu/2)M^{-\nu}.
\end{eqnarray*}
It remains to note that the maximum of r.h.s of \eqref{I_1(e^A)} is attained for \(A\in \{\pi/2M,3\pi/2M\}\)
and
\begin{eqnarray*}
\sup_A[I_2(e^A)+I_1(e^A)]\asymp \Gamma(\nu)|\cos(\pi\nu/2)|M^{-\nu+1}.
\end{eqnarray*}
The property \(q_{1,M}\in \mathcal{D}(\beta,\gamma,L) \) for all \(\beta<\nu-1\) and \(\gamma\in (0,2)\) with \(L\) depending on \(\gamma,\) follows from the identity \(\mathcal{M}[q_{1,M}](z)=\mathcal{M}[q](z)(1-\zeta_M)+\mathcal{M}[\rho_{M}](z)\mathcal{M}[q](z)\) and \eqref{mellin_rho}.
\end{proof}	
Let $T_{0,M}$ and $T_{1,M}$ be two random variables with densities $q_{0,M}$
and $q_{1,M}$ respectively. The the density of the r.v. $|W_{T_{i,M}}|,$
$i=0,1,$ is given by
\begin{align*}
p_{i,M}(x)  &  :=\frac{2}{\sqrt{2\pi}}\int_{0}^{\infty}\lambda^{-1/2}%
e^{-\frac{x^{2}}{2\lambda}}q_{i,M}(\lambda)d\lambda,\quad i=0,1.
\end{align*}
For the Mellin transform of $p_{i,M},$ we have
\begin{align}
\label{p_mellin}\mathcal{M}[p_{i,M}](z)  &  =\mathbb{E}\bigl[|W_{1}%
|^{z-1}\bigr]\mathbb{E}\bigl[T_{i,M}^{(z-1)/2}\bigr]\nonumber\\
&  =\mathbb{E}\bigl[|W_{1}|^{z-1}\bigr]\mathcal{M}[q_{i,M}%
]((z+1)/2)\nonumber\\
&  =\frac{2^{z/2}}{\sqrt{2\pi}}\Gamma(z/2)\mathcal{M}[q_{i,M}]((z+1)/2).
\end{align}

\begin{lem}
\label{l3_proof_low_log} The $\chi^{2}$-distance between the densities
$p_{0,M}$ and $p_{1,M}$ satisfies
\begin{align*}
\chi^{2}(p_{1,M}|p_{0,M}):=\int\frac{(p_{1,M}(x)-p_{0,M}(x))^{2}}{p_{0,M}%
(x)}dx\lesssim e^{-M\pi/2}, \quad M\to\infty.
\end{align*}

\end{lem}

\begin{proof}
First note that \(p_{0,M}(x)>0\) on \([0,\infty).\)
Since
\begin{eqnarray*}
\int_{0}^{1}\lambda^{-1/2}e^{-\frac{x^{2}}{2\lambda}}\log^{\nu-1}(1/\lambda)d\lambda&=&\int_{0}^{1}\lambda^{-1/2}e^{-\frac{x^{2}}{2\lambda}}\log^{\nu-1}(1/\lambda)d\lambda\\/y=1/\lambda,\lambda=1/y/&=&\int_{1}^{\infty}y^{-3/2}e^{-x^{2}y/2}\log^{\nu-1}(y)dy\\&=&\int_{x^{2}}^{\infty}x^{-2}(y/x^{2})^{-3/2}e^{-y/2}\log^{\nu-1}(y/x^{2})dy\\&=&x\int_{x^{2}}^{\infty}y{}^{-3/2}e^{-y/2}\log^{\nu-1}(y/x^{2})dy\lesssim e^{-x^{2}/2}
\end{eqnarray*}
and
\begin{eqnarray*}
\int_{1}^{\infty}\lambda^{-3/2}e^{-\frac{x^{2}}{2\lambda}}\log^{\nu-1}(\lambda)d\lambda&=&\int_{0}^{1}y^{-1/2}e^{-\frac{x^{2}}{2}y}\log^{\nu-1}(1/y)dy\\&\asymp&\frac{\Gamma(1/2)}{\sqrt{2}}x^{-1}\log^{\nu-1}(x^{2}).
\end{eqnarray*}
we have \(p_{0,M}(x)\gtrsim x^{-1},\) \(x\to \infty.\) Furthermore, due to \eqref{p_mellin} and the Parseval identity
\begin{multline}
\label{parseval_log}
\int_{0}^{\infty}x^{a-1}\left|p_{0,M}(x)-p_{1,M}(x)\right|^{2}dx=
\\
\frac{2^{-4+a}}{\pi}\int_{\gamma-\i\infty}^{\gamma+\i\infty}\mathcal{M}[q\vee \rho_M]\left(\frac{z+1}{2}\right)
\Gamma\left(\frac{z}{2}\right)\mathcal{M}[q\vee \rho_M]\left(\frac{a-z+1}{2}\right)\Gamma\left(\frac{a-z}{2}\right)dz,
\end{multline}
where \(\mathcal{M}[q\vee \rho_M](z)=\mathcal{M}[q](z)\mathcal{M}[\rho_M](z).\) Due to \eqref{mellin_rho}
\begin{eqnarray}
\label{mellin_rho_bound_log}
|\mathcal{M}[\rho_{M}](u+\i v)|\leq e^{\frac{(u-1)^{2}}{2}}\frac{\Phi(v+M)+\Phi(v-M)}{2}
\end{eqnarray}
with \(\Phi(v)=\int_{-\infty}^{v}e^{-\frac{x^{2}}{2}}dx.\) Combining \eqref{parseval_log} with properly chosen \(\gamma>0\),  \eqref{mellin_rho_bound_log} and Lemma ~\ref{lemma_gamma_asymp} (see Appendix), we derive
\begin{eqnarray*}
\chi^{2}(p_{1}|p_{0})&=&\int\frac{(p_{1}(x)-p_{0}(x))^{2}}{p_{0}(x)}dx\lesssim\int_{0}^{\infty}(p_{1}(x)-p_{0}(x))^{2}dx+\int_{0}^{\infty}x(p_{1}(x)-p_{0}(x))^{2}dx\\&\lesssim&\int_{-\infty}^{\infty}e^{-|v|\pi/2}\left(\Phi(v/2+M)+\Phi(v/2-M)\right)^{2}dv\lesssim e^{-M\pi/2}, \quad M\to \infty.
\end{eqnarray*}
\end{proof}
Fix some $\kappa\in(0,1/2).$ Due to Lemma~\ref{l3_proof_low_log}, the
inequality
\[
n\chi^{2}(p_{1,M}|p_{0,M})\leq\kappa
\]
holds for $M$ large enough, provided
\begin{align*}
M=\frac{2(1+\epsilon)}{\pi}\log(n)
\end{align*}
for arbitrary small $\epsilon>0.$ Hence Lemma~\ref{l2_proof_low_log} and
Theorem~\ref{ThmLowerBound} imply
\[
\inf_{\hat p_{n}}\sup_{p\in\mathcal{D}(\beta,\gamma,L)}\mathsf{P}%
_{p,n}\big(\|\hat p_{n}-p\|_{\infty}\ge c v_{n}\big)\ge\delta.
\]
for any $\beta<\nu-1,$ any $\gamma\in(0,2),$ some constants $c>0,$ $\delta>0$
and $v_{n}=\log^{-\beta}(n).$

\subsection{Proof of Proposition~\ref{uniform_upper_bound}}

It holds
\begin{align*}
p_{T,n}(x)-\mathbb{E}[p_{T,n}(x)]  &  =\frac{1}{\sqrt{\pi}}\int_{-1/h_{n}%
}^{1/h_{n}}x^{-\gamma-\i v}\frac{K(vh_{n})}{2^{\gamma+\i v}}\\
&  \times\frac{\bigl\{\mathcal{M}_{n}[p_{|X|}](2(\gamma+\i v)-1)-\mathcal{M}%
[p_{|X|}](2(\gamma+\i v)-1)\bigr\}}{\Gamma((\gamma+\i v)-1/2)}\,dv.
\end{align*}
Due to Proposition~\ref{ExpBounds}
\[
\sup_{x\geq0}\bigl\{x^{\gamma}|\mathbb{E}[p_{T,n}(x)]-p_{T}(x)|\bigr\}\leq
\frac{\Delta_{n}}{\sqrt{\pi n}}\int_{-1/h_{n}}^{1/h_{n}}\frac{A_{1}}%
{2^{\gamma}}\frac{\log(e+|v|)}{\Gamma((\gamma+\i v)-1/2)}\,dv
\]
with $\Delta_{n}=O_{a.s.}(1).$

\subsection{Proof of Proposition~\ref{asymp_norm}}

We have
\begin{align*}
\operatorname{Var}(Z_{n,1})  & =\frac{1}{\pi}\int_{-1/h_{n}}^{1/h_{n}}%
\int_{-1/h_{n}}^{1/h_{n}}\frac{x^{-2\gamma-{\i}(v-u)}}{2^{2\gamma
+{\i}(v-u)}}\frac{\operatorname{Cov}\left(  |X_{1}|^{2(\gamma+{\i}
v-1)},|X_{1}|^{2(\gamma+{\i} u-1)}\right)  }{2^{2\gamma+{\i}(v-u)}%
\Gamma(\gamma-1/2+{\i} v)\Gamma(\gamma-1/2-{\i} u)}\,du\,dv\\
& =\frac{1}{\pi}\int_{-1/h_{n}}^{1/h_{n}}\int_{-1/h_{n}}^{1/h_{n}}\frac
{1}{\left(  2x\right)  ^{2(\gamma-1)+{\i}(v-u)}}\frac{\mathcal{M}%
[p_{\left\vert X\right\vert }](4\gamma-3+2{\i}(v-u))}{\Gamma(\gamma
-1/2+{\i} v)\Gamma(\gamma-1/2-{\i} u)}\,dv\,du\\
& -\frac{1}{\pi}\left\vert \int_{-1/h_{n}}^{1/h_{n}}\frac{1}{\left(
2x\right)  ^{(\gamma+{\i} v-1)}}\frac{\mathcal{M}[p_{\left\vert X\right\vert
}](2\gamma-1+2{\i} v)}{\Gamma(\gamma-1/2+{\i} v)}\,dv\right\vert
^{2}=R_{1}-R_{2}.
\end{align*}
Note that
\begin{align*}
R_{2} &  \leq\frac{1}{(2x)^{2\left(  \gamma-1\right)  }}\left(  \int
_{-1/h_{n}}^{1/h_{n}}\left\vert \frac{\mathcal{M}[p_{\left\vert X\right\vert
}](2\gamma-1+2{{\i}}v)}{\Gamma(\gamma-1/2+{{\i}}v)}\,\right\vert
dv\right)  ^{2}\\
&  =\frac{1}{(2x)^{2\left(  \gamma-1\right)  }}\left\vert \int_{-1/h_{n}%
}^{1/h_{n}}\bigl|\mathcal{M}[p_{T}](\gamma+{{\i}}v)\bigr|\,dv\right\vert
^{2}<C<\infty
\end{align*}
and furthermore
\begin{align*}
R_{1} &  =\frac{1}{x^{2(\gamma-1)}\pi}\int_{-1/h_{n}}^{1/h_{n}}\int_{-1/h_{n}%
}^{1/h_{n}}\frac{1}{x^{{{\i}}(v-u)}}\frac{\Gamma(2\gamma-3/2+{{\i}
}(v-u))\mathcal{M}[p_{T}](2\gamma-1+{{\i}}(v-u))}{\Gamma(\gamma-1/2+{{\i}
}v)\Gamma(\gamma-1/2-{{\i}}u)}\,dv\,du\\
&  =\frac{1}{x^{2(\gamma-1)}\pi}\times I_{n}.
\end{align*}
Without loss of generality we may take $x=1$ (for $x\neq1$ the proof is
similar). Observe that for $v\in\mathbb{R},$
\begin{equation}
|\Gamma(\gamma-1/2+{{\i}}v)|\geq C_{1}1_{\left\vert v\right\vert \leq
2}+C_{2}1_{\left\vert v\right\vert >2}\left\vert v\right\vert ^{\gamma
-1}e^{-\pi\left\vert v\right\vert /2},\label{g1}%
\end{equation}
for some constants $C_{1}>0,$ $C_{2}>0$ (depending on $\gamma$), and that
\begin{equation}
|\Gamma(2\gamma-3/2+{{\i}}(v-u))|\leq D_{1}1_{\left\vert u-v\right\vert
\leq2}+1_{\left\vert u-v\right\vert >2}D_{2}\left\vert u-v\right\vert
^{2(\gamma-1)}e^{-\pi\left\vert u-v\right\vert /2}\label{g2}%
\end{equation}
for some $D_{1}>0,$ $D_{2}>0.$ Let $\rho_{n}=h_{n}^{-\alpha}$ for
$0<\alpha<1/2$. By the estimates (\ref{g1}) and (\ref{g2}), one can
straightforwardly derive that the integral
\[
I_{1,n,\rho_{n}}:=\int_{-1/h_{n}}^{1/h_{n}}\int_{-1/h_{n}}^{1/h_{n}%
}1_{\left\vert v-u\right\vert \geq\rho_{n}}\frac{\Gamma(2\gamma-3/2+{{\i}
}(v-u))\mathcal{M}[p_{T}](2\gamma-1+{{\i}}(v-u))}{\Gamma(\gamma-1/2+{{\i}
}v)\Gamma(\gamma-1/2-{{\i}}u)}\,dv\,du
\]
can be bounded from above as
\[
|I_{1,n,\rho_{n}}|\lesssim h_{n}^{-3\left\vert 1-\gamma\right\vert }%
e^{\pi\left(  \frac{1}{2h_{n}}-\rho_{n}/2\right)  }+h_{n}^{-2\left\vert
1-\gamma\right\vert -1}e^{\pi\left(  \frac{1}{h_{n}}-\rho_{n}/2\right)
}+h_{n}^{-4\left\vert 1-\gamma\right\vert -1}e^{\pi\left(  \frac{1}{h_{n}%
}-\rho_{n}\right)  }%
\]
for $n\rightarrow\infty.$ Similarly
\begin{multline*}
\int_{-1/h_{n}}^{1/h_{n}}\int_{-1/h_{n}}^{1/h_{n}}1_{\left\vert u\right\vert
\leq\frac{1}{h_{n}}-\rho_{n}}1_{\left\vert v-u\right\vert \leq\rho_{n}}%
\frac{\Gamma(2\gamma-3/2+{{\i}}(v-u))\mathcal{M}[p_{T}](2\gamma-1+{{\i}
}(v-u))}{\Gamma(\gamma-1/2+{{\i}}v)\Gamma(\gamma-1/2-{{\i}}u)}\,dv\,du\\
=O\Bigl(h_{n}^{-l}e^{\pi\left(  \frac{1}{h_{n}}-\rho_{n}\right)  }\Bigr)
\end{multline*}
and
\begin{multline*}
\int_{-1/h_{n}}^{1/h_{n}}\int_{-1/h_{n}}^{1/h_{n}}1_{\left\vert v\right\vert
\leq\frac{1}{h_{n}}-\rho_{n}}1_{\left\vert v-u\right\vert \leq\rho_{n}}%
\frac{\Gamma(2\gamma-3/2+{{\i}}(v-u))\mathcal{M}[p_{T}](2\gamma-1+{{\i}
}(v-u))}{\Gamma(\gamma-1/2+{{\i}}v)\Gamma(\gamma-1/2-{{\i}}u)}\,dv\,du\\
=O\Bigl(h_{n}^{-l}e^{\pi\left(  \frac{1}{h_{n}}-\rho_{n}\right)  }\Bigr)
\end{multline*}
for some $l>0.$ Hence
\begin{align*}
I_{2,n,\rho_{n}}:=  & \int_{-1/h_{n}}^{1/h_{n}}\int_{-1/h_{n}}^{1/h_{n}%
}1_{\left\vert v-u\right\vert \leq\rho_{n}}\frac{\Gamma(2\gamma-3/2+{{\i}
}(v-u))\mathcal{M}[p_{T}](2\gamma-1+{{\i}}(v-u))}{\Gamma(\gamma-1/2+{{\i}
}v)\Gamma(\gamma-1/2-{{\i}}u)}\,dv\,du\\
& =\int_{-1/h_{n}}^{1/h_{n}}\int_{-1/h_{n}}^{1/h_{n}}1_{\left\vert
u\right\vert \geq\frac{1}{h_{n}}-\rho_{n}}1_{\left\vert v\right\vert \geq
\frac{1}{h_{n}}-\rho_{n}}1_{\left\vert v-u\right\vert \leq\rho_{n}}\\
& \times\frac{\Gamma(2\gamma-3/2+{{\i}}(v-u))\mathcal{M}[p_{T}%
](2\gamma-1+{{\i}}(v-u))}{\Gamma(\gamma-1/2+{{\i}}v)\Gamma(\gamma
-1/2-{{\i}}u)}\,dv\,du+O\Bigl(h_{n}^{-l}e^{\pi\left(  \frac{1}{h_{n}}%
-\rho_{n}\right)  }\Bigr)\\
& =:I_{3,n,\rho_{n}}+O\Bigl(h_{n}^{-l}e^{\pi\left(  \frac{1}{h_{n}}-\rho
_{n}\right)  }\Bigr).
\end{align*}
Now let us study the asymptotic behaviour of the integral $I_{3,n,\rho_{n}}.$
To this end, we will use the Stirling formula
\[
\Gamma(\gamma-1/2+{{\i}}v)=\left(  \gamma-1/2+iv\right)  ^{\gamma
-1+iv}e^{-\gamma+1/2-iv}\sqrt{2\pi}(1+O(\left\vert v\right\vert ^{-1})).
\]
First consider the integrand of $I_{3,n,\rho_{n}}$ in the case $u,v\rightarrow
+\infty,$ where
\begin{align*}
\Gamma(\gamma-1/2+{{\i}}v)\Gamma(\gamma-1/2-{{\i}}u) &  =2\pi\exp\left[
iv\log v-iu\log u-i\left(  v-u\right)  \right]  \\
&  \times\exp\left[  -\frac{\pi}{2}\left(  u+v\right)  +\left(  \gamma
-1\right)  \left(  \log v+\log u\right)  \right]  (1+O(1/u)+O(1/v)).
\end{align*}
Then on the set
\[
\left\{  \left\vert u\right\vert \geq\frac{1}{h_{n}}-\rho_{n}\right\}
\cap\left\{  \left\vert v\right\vert \geq\frac{1}{h_{n}}-\rho_{n}\right\}
\cap\left\{  \left\vert v-u\right\vert \leq\rho_{n}\right\}  \cap\left\{
v\geq0,u\geq0\right\}
\]
we define $u=1/h_{n}-r,$ $v=1/h_{n}-s$ with $0<r,s<\rho_{n},$ $\left\vert
r-s\right\vert <\rho_{n}$ to get%

\begin{align*}
\Gamma(\gamma-1/2+{{\i}}v)\Gamma(\gamma-1/2-{{\i}}u) &  =2e^{i\left(
1/h_{n}-s\right)  \log\left(  1/h_{n}-s\right)  -i\left(  1/h_{n}-r\right)
\log\left(  1/h_{n}-r\right)  -i\left(  r-s\right)  }\\
&  \times h_{n}^{-2\left(  \gamma-1\right)  }\exp\left[  -\pi/h_{n}\right]
\exp\left[  (r+s)\pi\right]  \\
&  \times(1+O(h_{n}))(1+O(\rho_{n}h_{n})).
\end{align*}
Note that due to the choice of $\rho_{n},$ $\rho_{n}h_{n}\downarrow0$ and
$\rho_{n}^{2}h_{n}\downarrow0.$ Using the asymptotic expansion
\[
\left(  1/h_{n}-s\right)  \log\left(  1/h_{n}-s\right)  -\left(
1/h_{n}-r\right)  \log\left(  1/h_{n}-r\right)  -\left(  r-s\right)  =\left(
r-s\right)  \log\left(  1/h_{n}\right)  +O(\rho_{n}^{2}h_{n}),
\]
we derive
\begin{align*}
\Gamma(\gamma-1/2+{{\i}}v)\Gamma(\gamma-1/2-{{\i}}u)  & =2\pi
h_{n}^{-2\left(  \gamma-1\right)  }\exp\left[  -\pi/h_{n}\right]  \exp\left[
(r+s)\pi\right]  \\
& \times\exp\left[  i\left(  r-s\right)  \log\left(  1/h_{n}\right)  \right]
(1+O(\rho_{n}^{2}h_{n})).
\end{align*}
Analogously, on the set
\[
\left\{  \left\vert u\right\vert \geq\frac{1}{h_{n}}-\rho_{n}\right\}
\cap\left\{  \left\vert v\right\vert \geq\frac{1}{h_{n}}-\rho_{n}\right\}
\cap\left\{  \left\vert v-u\right\vert \leq\rho_{n}\right\}  \cap\left\{
v\leq0,u\leq0\right\}
\]
we define $u=-1/h_{n}+r,$ $v=-1/h_{n}+s,$ with $0<r,s<\rho_{n},$ $\left\vert
r-s\right\vert <\rho_{n},$ to get
\begin{align*}
\Gamma(\gamma-1/2+{{\i}}v)\Gamma(\gamma-1/2-{{\i}}u)  & =2\pi
h_{n}^{-2\left(  \gamma-1\right)  }\exp\left[  -\pi/h_{n}\right]  \exp\left[
(r+s)\pi\right]  \\
& \times\exp\left[  -i\left(  r-s\right)  \log\left(  1/h_{n}\right)  \right]
(1+O(\rho_{n}^{2}h_{n})).
\end{align*}
Hence the integral $I_{3,n,\rho_{n}}$ can be decomposed as follows
\[
I_{3,n,\rho_{n}}=:\frac{h_{n}^{2\left(  \gamma-1\right)  }}{\pi}\exp\left[
\pi/h_{n}\right]  \left\{  \Re\lbrack I_{4,n,\rho_{n}}]+O(\rho_{n}^{2}%
h_{n})\right\},
\]
where
\begin{align*}
I_{4,n,\rho_{n}}  & =\int\int1_{0\leq r\leq\rho_{n}}1_{0\leq s\leq\rho_{n}%
}1_{\left\vert r-s\right\vert \leq\rho_{n}}\exp\left[  -(r+s)\pi\right]
\Gamma(2\gamma-3/2+{{\i}}(r-s))\\
& \times\mathcal{M}[p_{T}](2\gamma-1+{{\i}}(r-s))\exp\left[  i\left(
s-r\right)  \log\left(  1/h_{n}\right)  \right]  drds\\
& =\int_{0}^{\rho_{n}}e^{-2v\pi}R_{n}(v)dv
\end{align*}
with
\[
R_{n}(v)=\int1_{0\leq u\leq\rho_{n}-v}e^{-u\pi}\Gamma(2\gamma-3/2+{{\i}
}u)\mathcal{M}[p_{T}](2\gamma-1+{{\i}}u)e^{iu\log\left(  1/h_{n}\right)
}du.
\]
Using the saddle point method (see, e.g.,  de Bruijn, \cite{de1970asymptotic}), it is easy to show that
\begin{align*}
R_{n}(v)  & =e^{i\pi/2}\Gamma(2\gamma-3/2)\mathcal{M}[p_{T}](2\gamma
-1)\log^{-1}\left(  1/h_{n}\right)  \\
& +e^{i\pi}\left[  \left.  \frac{d}{du}\left(  \Gamma(2\gamma
-3/2+iu)\mathcal{M}[p_{T}](2\gamma-1+iu)\right)  \right\vert _{u=0}\right]
\log^{-2}\left(  1/h_{n}\right)  +O(\log^{-3}\left(  1/h_{n}\right)  )
\end{align*}
uniformly in $v.$ As a result
\[
\Re\lbrack I_{4,n,\rho_{n}}]=\left[  \left.  \frac{d}{du}\left(
\Gamma(2\gamma-3/2+iu)\mathcal{M}[p_{T}](2\gamma-1+iu)\right)  \right\vert
_{u=0}\right]  \log^{-2}\left(  1/h_{n}\right)  +O(\log^{-3}\left(
1/h_{n}\right)  ).
\]
Combining all above estimates, we finally get%

\begin{align}
\operatorname{Var}(Z_{n,1})  & =\frac{h_{n}^{2\left(  \gamma-1\right)  }}%
{\pi^{2}}\log^{-2}\left(  1/h_{n}\right)  \exp\left[  \pi/h_{n}\right]
\label{var_asymp}\\
& \times\left\{  \left[  \left.  \frac{d}{du}\left(  \Gamma(2\gamma
-3/2+iu)\mathcal{M}[p_{T}](2\gamma-1+iu)\right)  \right\vert _{u=0}\right]
\right.  \nonumber\\
& \left.  +O(\log^{-1}\left(  1/h_{n}\right)  )+O(\rho_{n}^{2}h_{n}\log
^{2}\left(  1/h_{n}\right)  )+O\Bigl(e^{-\pi\rho_{n}/2}\log^{2}\left(
1/h_{n}\right)  \Bigr)\right\}  .\nonumber
\end{align}
Using the decomposition
\eqref{var_asymp}, the Lyapounov condition for some $\delta>0$
\[
\frac{\mathbb{E}|Z_{n,1}-\mathbb{E}Z_{n,1}|^{2+\delta}}{n^{\delta
/2}[\operatorname{Var}(Z_{n,1})]^{1+\delta/2}}\rightarrow0,\quad
n\rightarrow\infty
\]
is easy to verify, since $\mathbb{E}Z_{n,1}\rightarrow p_{T}(x).$

\subsection{Proof of Proposition~\ref{arc}}

Let $\theta_{\max}$ be such that $A=\tan\theta_{\max}.$ At the arc
$K_{R}:w=R\,e^{i\theta},$ $-\theta_{\max}<\theta<\theta_{\max},$ it holds that%
\begin{align*}
\left\vert \int_{K_{R}}w^{z-1}\mathcal{L}[p_{T}](w)dw\right\vert  &  \leq
R\theta_{\max}\cdot R^{\operatorname{Re}z-1}\int e^{-xR\cos\theta_{\max}}%
p_{T}(x)dx\\
&  \leq B\theta_{\max}R^{\operatorname{Re}z}\int\,e^{-xR\cos\theta_{\max}%
}dx=B\theta_{\max}\frac{R^{\operatorname{Re}z-1}}{\cos\theta_{\max}%
}\rightarrow0,
\end{align*}
for $0<\operatorname{Re}z<1,$ where $\sup_{x>0}p_{T}(x)\leq B.$

\subsection{Proof of Proposition \ref{polR}}

By (\ref{FpX}) we derive for the bias of $p_{T,n}(x),$ $x>0,$
\begin{gather*}
|\mathbb{E}[p_{T,n}(x)]-p_{T}(x)|=\left\vert \frac{1}{2\pi}\int_{-U_{n}%
}^{U_{n}}\frac{\mathrm{E}\left[  \Phi_{n}(1-\gamma-\i v,X_{1})\right]
}{\Gamma(1-\gamma-\i v)}x^{-\i v}dv-\int_{-\infty}^{\infty}\mathcal{M}%
[p_{T}](\gamma+\i v)x^{-\gamma-\i v}dv\right\vert \\
\leq\left\vert \frac{1}{2\pi}\int_{-U_{n}}^{U_{n}}\frac{\int_{A_{n}}^{\infty
}\left[  \psi(\lambda)\right]  ^{-\gamma-\i v}\mathcal{F}[p_{X}](\lambda
)\psi^{\prime}(\lambda)d\lambda}{\Gamma(1-\gamma-\i v)}x^{-\gamma-\i
v}dv\right\vert +\frac{\left\vert x\right\vert ^{-\gamma}}{2\pi}%
\int_{\{|v|>U_{n}\}}\left\vert \mathcal{M}[p_{T}](\gamma+\i v)\right\vert dv\\
=:(\ast)_{1}+(\ast)_{2}%
\end{gather*}
Similar to the proof of Theorem \ref{sep_conv_rates} we have,
\[
(\ast)_{2}\leq\frac{\left\vert x\right\vert ^{-\gamma}}{2\pi}e^{-\beta U_{n}%
}\int_{\{|v|>U_{n}\}}\left\vert \mathcal{M}[p_{T}](\gamma+\i v)\right\vert
e^{\beta|v|}dv\leq e^{-\beta U_{n}}\frac{\left\vert x\right\vert ^{-\gamma}%
L}{2\pi},
\]
and by Lemma \ref{dif} and \eqref{gamma_asymp}%
\begin{align*}
(\ast)_{1}  &  \lesssim\frac{|x|^{-\gamma}}{2\pi}\int_{-U_{n}}^{U_{n}}%
\frac{\int_{A_{n}}^{\infty}\lambda^{-2\gamma+1}\left\vert \mathcal{F}%
[p_{X}](\lambda)\right\vert d\lambda}{\left\vert \Gamma(1-\gamma-\i
v)\right\vert }dv\\
&  \lesssim|x|^{-\gamma}U_{n}^{\gamma-1/2}e^{U_{n}\pi/2}\int_{A_{n}}^{\infty
}\frac{\lambda^{-\epsilon}}{\lambda^{2\gamma-1-\epsilon}}\left\vert
\mathcal{F}[p_{X}](\lambda)\right\vert d\lambda\lesssim|x|^{-\gamma}%
\frac{U_{n}^{\gamma-1/2}e^{U_{n}\pi/2}}{A_{n}^{\epsilon}}.
\end{align*}
As for the variance
\begin{align}
\mathrm{Var}(p_{T,n}(x))  &  =\frac{1}{(2\pi)^{2}n}\mathrm{Var}\left[
\int_{-U_{n}}^{U_{n}}\frac{\Phi_{n}(1-\gamma-\i v,X_{1})}{\Gamma(1-\gamma-\i
v)}x^{-\gamma-\i v}dv\right] \nonumber\\
&  \leq\frac{1}{(2\pi)^{2}n}|x|^{-2\gamma}\left[  \int_{-U_{n}}^{U_{n}}%
\frac{\sqrt{\mathrm{Var}[\Phi_{n}(1-\gamma-\i v,X_{1})]}}{\left\vert
\Gamma(1-\gamma-\i v)\right\vert }dv\right]  ^{2}, \label{hu}%
\end{align}
where%
\begin{align*}
\sqrt{\mathrm{Var}[\Phi_{n}(1-\gamma-\i v,X_{1})]}  &  \leq\int_{0}^{A_{n}%
}\sqrt{\mathrm{Var}[\left[  \psi(\lambda)\right]  ^{-\gamma-\i v}e^{\i
X_{1}\lambda}\psi^{\prime}(\lambda)]}d\lambda\\
&  =\int_{0}^{A_{n}}\left\vert \psi(\lambda)\right\vert ^{-\gamma}\left\vert
\psi^{\prime}(\lambda)\right\vert \sqrt{\mathrm{Var}[e^{\i X_{1}\lambda}%
]}d\lambda.
\end{align*}
Due to Lemma \ref{dif} we have
\[
\int_{1}^{A_{n}}\left\vert \psi(\lambda)\right\vert ^{-\gamma}\left\vert
\psi^{\prime}(\lambda)\right\vert \sqrt{\mathrm{Var}[e^{\i X_{1}\lambda}%
]}d\lambda\lesssim\int_{1}^{A_{n}}\lambda^{(1-2\gamma)}d\lambda\leq C_{0}%
\frac{A_{n}^{2\left(  1-\gamma\right)  }}{1-\gamma}%
\]
and in any case of Lemma \ref{dif} it holds%
\[
\int_{0}^{1}\left\vert \psi(\lambda)\right\vert ^{-\gamma}\left\vert
\psi^{\prime}(\lambda)\right\vert \sqrt{\mathrm{Var}[e^{\i X_{1}\lambda}%
]}d\lambda\leq\int_{0}^{1}\left\vert \psi(\lambda)\right\vert ^{-\gamma
}\left\vert \psi^{\prime}(\lambda)\right\vert d\lambda\leq\frac{C_{1}%
}{1-\gamma}%
\]
for some natural constant $C_{0},C_{1}>0.$ Hence from (\ref{hu}) we get by
(\ref{gamma_asymp}),
\[
|x|^{2\gamma}\mathrm{Var}(p_{T,n}(x))\leq\frac{1}{(2\pi)^{2}n}\left(
CU_{n}^{\gamma-1/2}e^{U_{n}\pi/2}\frac{A_{n}^{2\left(  1-\gamma\right)  }%
}{1-\gamma}\right)  ^{2}=:(\ast)_{3},
\]
and by gathering $(\ast)_{1},$ $(\ast)_{2},$ and $(\ast)_{3},$%
\[
\sqrt{\mathrm{E}\left[  x^{2\gamma}\left\vert p_{n}(x)-p(x)\right\vert
^{2}\right]  }\lesssim\frac{C}{2\pi\left(  1-\gamma\right)  \sqrt{n}}%
U_{n}^{\gamma-1/2}e^{U_{n}\pi/2}A_{n}^{2\left(  1-\gamma\right)  }+\frac
{U_{n}^{\gamma-1/2}e^{U_{n}\pi/2}}{A_{n}^{\epsilon}}+e^{-\beta U_{n}}.
\]
Next, the choices (\ref{ch1}) and (\ref{ch2}) lead to the desired result.

\section{Appendix}

\begin{prop}
\label{ExpBounds} Let $Z_{j},$ $j=1,\ldots, n, $ be a sequence of independent
identically distributed random variables. Fix some $u>0$ and define
\begin{align*}
\varphi_{n}(v):=\frac{1}{n}\sum_{j=1}^{n}\exp\left\{  \left(  u +\i v \right)
Z_{j}\right\}  , \quad v\in\mathbb{R}.
\end{align*}
Furthermore let $w$ be a positive monotone decreasing Lipschitz function on
$\mathbb{R}_{+}$ such that
\begin{equation}
\label{decreasing_w}0<w(z)\leq\frac{1}{\sqrt{\log(e+|z|)}}, \quad
z\in\mathbb{R}_{+}.
\end{equation}
Suppose that $\mathbb{E} \bigl[ e^{pu Z}\bigr]<\infty$ and $\mathbb{E}
\bigl[ |Z|^{p}\bigr]<\infty$ for some $p>2.$ Then with probability $1$
\begin{align}
\label{MINEQ}\left\|  \varphi_{n}- \varphi\right\|  _{L_{\infty}%
(\mathbb{R},w)}=O\left(  \sqrt{\frac{\log n}{n}}\right)  .
\end{align}

\end{prop}

\begin{proof}
% The proof is almost a copy (with minor changes) of the proof of Proposition A.3 from \cite{panov2013a}.
Fix a sequence \(\Xi_{n}\to \infty\) as \(n\to \infty.\) Denote
\begin{eqnarray*}
\mathcal{W}_{n}^{1}(v) &:=&
\frac{w(v) }{n} \;
	\sum_{j=1}^{n}
	\Bigl(
			e^{(u+\i v) Z_{j}}
			\mathbb{I}\left\{ e^{u Z_{j }} < \Xi_{n} \right\}
			-		
			\E \left[	
				e^{(u+\i v) Z}
				\mathbb{I}\left\{ e^{u Z} < \Xi_{n} \right\}
			\right]
	\Bigr),\\
\mathcal{W}_{n}^{2}(v) &:=&
\frac{w(v) }{n} \;
	\sum_{j=1}^{n}
	\Bigl(
			e^{(u+\i v) Z_{j}}
			\mathbb{I}\left\{ e^{u Z_{j }} \geq \Xi_{n} \right\}
			-		
			\E \left[	
				e^{(u+\i v) Z}
				\mathbb{I}\left\{ e^{u Z} \geq \Xi_{n} \right\}
			\right]
	\Bigr),
\end{eqnarray*}
where \(Z\) is a random variable with the same distribution as \(Z_{1}\). The main idea of the proof is to show that
\begin{eqnarray}
\label{aim1}
|\mathcal{W}_{n}^{1}(v)|&=&O_{a.s.}\left(\sqrt{\frac{\log n}{n}}\right),\\
\label{aim2}
|\mathcal{W}_{n}^{2}(v) |&=&O_{a.s.}\left(\sqrt{\frac{\log n}{n}}\right)
\end{eqnarray}
under a proper choice of the sequence \(\Xi_{n}.\)
\\
\textbf{Step 1.}  The aim of the first step is to show \eqref{aim1}. Consider the sequence \( A_{k}=e^{k},\, k\in \mathbb{N} \) and cover each
interval \( [-A_{k},A_{k}] \) by
\( M_{k}=\left(\lfloor 2A_{k}/\gamma \rfloor +1  \right) \) disjoint small
intervals \( \Lambda_{k,1},\ldots,\Lambda_{k,M_{k}} \) of
the length \( \gamma. \) Let \( v_{k,1},\ldots, v_{k,M_{k}} \) be the centers
of these intervals. We have for any natural \( K>0 \)
\begin{multline*}
\max_{k=1,\ldots,K}\sup_{A_{k-1}<| v |\leq A_{k}}|\mathcal{W}_{n}^{1}(v)|\leq
\max_{k=1,\ldots,K}\max_{1\leq m \leq M_{k}}
\sup_{v\in \Lambda_{k,m}}|\mathcal{W}_{n}^{1}(v)-\mathcal{W}_{n}^{1}(v_{k,m})|
\\
+\max_{k=1,\ldots,K}\max_{\Bigl\{\substack{ 1\leq m \leq M_{k}:\\
| v_{k,m} |>A_{k-1}}\Bigr\}}|\mathcal{W}_{n}^{1}(v_{k,m})|.
\end{multline*}
Hence for any positive \(\lambda\),
\begin{multline}
\label{DEC1}
\P\left( \max_{k=1,\ldots,K}\sup_{A_{k-1}< | v |\leq A_{k}}|\mathcal{W}_{n}^{1}(v)|>\lambda \right)
\leq
\P\left(\sup_{| v_{1}-v_{2} |<\gamma}|\mathcal{W}_{n}^{1}(v_{1})-\mathcal{W}_{n}^{1}(v_{2})|>\lambda/2\right)
\\
+
\sum_{k=1}^{K}\sum_{\Bigl\{\substack{ 1\leq m \leq M_{k}:\\
| v_{k,m} |>A_{k-1}}\Bigr\}}\P(|\mathcal{W}_{n}^{1}(v_{k,m})|>\lambda/2).
\end{multline}
We proceed with the first summand in \eqref{DEC1}. It holds for any \( v_{1},v_{2}\in \mathbb{R} \)
\begin{eqnarray}
\label{WNDIFF}
\nonumber
|\mathcal{W}_{n}^{1}(v_{1})-\mathcal{W}_{n}^{1}(v_{2})|&\leq&
2 \: \Xi_{n}
\left|
	w( v_{1})-w( v_{2} )
\right|
+\frac{1}{n}\sum_{j=1}^{n}
\Bigl[
	\left|
		e^{(u+\i v_{1}) Z_{j}} - e^{(u+\i v_{2}) Z_{j}}
	\right| I\left\{ e^{u Z_{j}} < \Xi_{n} \right\}
\Bigr]\\
&&\nonumber+
\Bigl|
\E \left[
		\left(
			 e^{(u+\i v_{1}) Z} - e^{(u+\i v_{2}) Z}
	 	\right)
	 I\left\{
	 		e^{u Z} < \Xi_{n}
		\right\}
	\right]
\Bigr|
\\
&\leq& \left| v_{1}-v_{2} \right|\:\Xi_{n}\:
\left[ 2 \:
L_{w}+\frac{1}{n}\sum_{j=1}^{n}| Z_{j}|+\E| Z| \right],
\end{eqnarray}
where \( L_{\omega } \) is the Lipschitz constant of \( w\) and \(Z\) is a random variable distributed as \(Z_{1}\). Next, the Markov inequality implies
\begin{eqnarray*}
\P\left\{ \frac{1}{n}\sum_{j=1}^{n}\Bigl[|  Z_{j} |-\E| Z |\Bigr]>
c \right\}\leq c^{-p}n^{-p}\:\E\left| \sum_{j=1}^{n}\Bigl[| Z_{j} |-\E| Z | \Bigr] \right|^{p}
\end{eqnarray*}
for any \( c >0. \) Note that
\begin{eqnarray*}
\E\left| \sum_{j=1}^{n}\Bigl[
	|  Z_{j} |-\E| Z |
\Bigr] \right|^{p}\leq c_{p} n^{p/2},
\end{eqnarray*}
for some constant \( c_{p} \) depending on \( p\)
and we obtain from \eqref{WNDIFF}
\begin{equation*}
\label{LINEQ}
\P\Bigl\{
	\sup_{| v_{1}-v_{2}|<\gamma}|\mathcal{W}_{n}^{1}(v_{1})-\mathcal{W}_{n}^{1}(v_{2})|>2\gamma \Xi_{n} (L_{\omega}+\E|Z|+c)\Bigr\}
	\leq
	 C_{p} \,c^{-p} n^{-p/2}.
\end{equation*}
Hence if \(\gamma \Xi_{n}\geq 1\) and \(\lambda\geq 4(L_{\omega}+\E|Z|+c)\) we get
Now we turn to the second term on the right-hand side of  \eqref{DEC1}.
Applying  the Bernstein inequality, we get
\begin{equation*}
%\label{LINEQ}
\P\Bigl\{
	\sup_{| v_{1}-v_{2}|<\gamma}|\mathcal{W}_{n}^{1}(v_{1})-\mathcal{W}_{n}^{1}(v_{2})|>\lambda/2\Bigr\}
	\leq
	 C_{p} \,c^{-p} n^{-p/2}.
\end{equation*}
\begin{eqnarray*}
\P\left(|\Re\left[ \mathcal{W}_{n}^{1}(v_{k,m}) \right]|>\lambda/4\right) \leq   \exp\left(
-\frac{
	\lambda^{2}n
}{
	32(\Xi_{n} w(A_{k-1}) \lambda/3 +w^2(A_{k-1})\, \E[e^{2uZ}])}
\right).
\end{eqnarray*}
Similarly,
\begin{eqnarray*}
\P\left(|\Im\left[ \mathcal{W}_{n}^{1}(v_{k,m}) \right]|>\lambda/4\right)
\leq  \exp\left(
-\frac{
	\lambda^{2}n
}{
	32(\Xi_{n} w(A_{k-1}) \lambda/3 +w^2(A_{k-1})\, \E[e^{2uZ}])}
\right).
\end{eqnarray*}
Therefore
\begin{eqnarray*}
\sum_{\{ | v_{k,m} |>A_{k-1} \}}\P(|\mathcal{W}_{n}^{1}(v_{k,m})|>\lambda/2)\leq
\left(\lfloor 2A_{k}/\gamma \rfloor +1  \right)\exp\left(
-\frac{
	\lambda^{2}n
}{
	32(\Xi_{n} w(A_{k-1}) \lambda/3 +w^2(A_{k-1})\, \E[e^{2uZ}])}
\right).
\end{eqnarray*}
Set now \(\gamma= \sqrt{(\log n)/n},\)  \(\lambda = \zeta \sqrt{(\log n)/n}\) and \(\Xi_n=\sqrt{n/\log (n)},\) then
\begin{eqnarray*}
\sum_{\{ | v_{k,m} |>A_{k-1} \}}\P(|\mathcal{W}_{n}^{1}(v_{k,m})|>\lambda/2)
&\lesssim &
A_k\,\sqrt{\frac{n}{\log(n) }}
\exp\left(
-\frac{
	\lambda^{2}n
}{
	32\bigl(\Xi_{n} w(A_{k-1}) \lambda/3 +w^2(A_{k-1})\, \E[e^{2uZ}]\bigr)}
\right)
\\
&\lesssim &
\sqrt{\frac{n}{\log(n) }}
\exp\left(-k+k\left[1-
	\frac{\zeta^{2}\log (n)}{32(1+\E[e^{2uZ}])}\right]
\right).
\end{eqnarray*}
Assuming that \( \zeta^2\geq 32\theta (1+\E[e^{2uZ}])\) for some \(\theta>1\), we arrive at
\begin{eqnarray*}
\sum_{k=2}^{\infty}\sum_{\{ |v_{k,m} |>A_{k-1} \}}\P(|\mathcal{W}_{n}(v_{k,m})|>\lambda/2)
& \lesssim &   e^{-k}\frac{n^{1/2-\theta}}{\sqrt{\log(n)}}, \quad n\to \infty
\end{eqnarray*}
\textbf{Step 2}. Now we turn to \eqref{aim2}.
Consider the sequence
\[
R_{n}(v) :=
\frac{1 }{n} \;
	\sum_{j=1}^{n}
			e^{(u+\i v) Z_{j}}
			\mathbb{I}\left\{ e^{u Z_{j }} \geq \Xi_{n} \right\}.
\]
By the Markov inequality we get for any \(p>1\)
\[
	\left|
		\E \left[
			R_{n}(u)
		\right]
	\right|
	\leq \E \left[
		e^{u Z_{j}}
	\right]
	\P \left\{ e^{u Z_{j }} \geq \Xi_{n} \right\}
	\leq
	\Xi_{n}^{-p}
	\;\;
	\E \left[
		e^{u Z_{j}}
	\right]
	\;
	\E
	\left[
		e^{u p Z_{j}}
	\right] =  o\Bigl(\sqrt{(\log n)/ n}\Bigr)
\]
Set \(\eta_{k}=2^{k}, k=1,2,\ldots\), then it holds for any \(p>2\)
\[
\sum_{k=1}^{\infty}\P\Bigl\{ \max_{j=1,\ldots,\eta_{k+1}}   e^{u Z_{j}} \geq \Xi_{\eta_{k}}\Bigr\}
\leq
	\sum_{k=1}^{\infty} \eta_{k+1} \P\{ e^{u Z} \geq \Xi_{\eta_{k}}\}
	\leq
	\E [e^{p u Z}]
	\sum_{k=1}^{\infty} \eta_{k+1} \Xi_{\eta_{k}}^{-p}<\infty.
	\]
By the Borel-Cantelli lemma,
\[
	\P\Bigl\{ \max_{j=1,\ldots,\eta_{k+1}}   e^{u Z_{j}} \geq \Xi_{\eta_{k}} \quad\mbox {for infinitely many } k\Bigr\} = 0.
\]
From here it follows that \(R_{n}(u)- \E R_{n}(u) = o\Bigl(\sqrt{(\log n)/ n}\Bigr)\). This completes the proof.
\end{proof}

\begin{lem}
\label{dif} Let $(L_{t},\,t\geq0)$ be a L\'{e}vy process with the triplet
$(\mu,\sigma^{2},\nu).$ Suppose that $\int_{\{|x|>1\}}|x|\nu(dx)<\infty,$ and
that $\sigma$ and $\nu$ are not both zero. It then holds for $\psi
(u)=-\log(\mathbb{E}(\exp(\i uL_{t})))$
\begin{equation}
(i):\text{ }\left\vert \psi(u)\right\vert \lesssim u^{2}\text{ \ \ and
\ \ }(ii):\text{ }\left\vert \psi^{\prime}(u)\right\vert \lesssim u,\text{
\ \ }u\rightarrow\infty. \label{uinf}%
\end{equation}
Further, if
\begin{equation}
d=\mu+\int_{\{|x|>1\}}x\nu(dx)\neq0 \label{dr}%
\end{equation}
we have
\begin{equation}
(i):\text{ }\left\vert \psi(u)\right\vert \gtrsim u\text{ \ \ and
\ }(ii):\text{ \ }\left\vert \psi^{\prime}(u)\right\vert \lesssim1,\text{
\ \ }u\downarrow0. \label{u0}%
\end{equation}
If $d=0$ we have in the case $\nu(\{|x|>1\}\cap dx)\equiv0,$
\begin{equation}
(i):\text{ }\left\vert \psi(u)\right\vert \gtrsim u^{2},\text{ \ \ and
\ \ }(ii):\text{ }\left\vert \psi^{\prime}(u)\right\vert \lesssim u,\text{
\ \ }u\downarrow0, \label{u01}%
\end{equation}
and in the case $\nu(\{|x|>1\}\cap dx)\neq0,$%
\begin{equation}
(i):\text{ }\left\vert \psi(u)\right\vert \gtrsim u,\text{ \ \ and
\ \ }(ii):\text{ }\left\vert \psi^{\prime}(u)\right\vert =o(1),\text{
\ \ }u\downarrow0. \label{u02}%
\end{equation}

\end{lem}

\begin{proof} In general we have%
\begin{equation}
\psi(u)=-\i u\mu+\frac{u^{2}\sigma^{2}}{2}+\int_{\mathbb{R}%
}(1-e^{\i ux}+\i ux1_{\left\vert x\right\vert \leq1})\nu(dx),
\label{as1}%
\end{equation}
where
\begin{align}
\int_{\mathbb{R}}(1-e^{\i ux}+\i ux1_{\left\vert
x\right\vert \leq1})\nu(dx)  &  =u^{2}\int_{\{|x|\leq1\}}\frac{1-e^{\i
ux}+\i ux}{\left(  ux\right)  ^{2}}x^{2}\nu(dx)\label{as}\\
&  +\int_{\{|x|>1\}}\left(  1-e^{\i ux}\right)  \nu(dx).\nonumber
\end{align}
Note that
\[
0<c_{1}<\frac{\left\vert 1-e^{\i y}+\i y\right\vert }{y^{2}%
}<c_{2}\text{ \ for \ }y\in\mathbb{R},
\]
with $0<c_{1}<c_{2},$ and that%
\[
\int_{\{|x|>1\}}\left(  1-e^{\i ux}\right)  xv(dx)\longrightarrow
\int_{\{|x|>1\}}xv(dx)\text{ \ \ for }u\rightarrow\infty
\]
by Riemann-Lebesgue. This yields (\ref{uinf})-$(i).$ It is not difficult to
show by standard arguments that due to the integrability condition we have
\[
\psi^{\prime}(u)=-\i\mu+u\sigma^{2}-\i\int_{\mathbb{R}%
}(e^{\i ux}-1_{\left\vert x\right\vert \leq1})x\nu(dx).
\]
Next, (\ref{uinf})-$(ii)$ follows by observing that%
\[
\int_{\{|x|\leq1\}}(e^{\i ux}-1)xv(dx)=u\int_{\{|x|\leq1\}}%
\frac{e^{\i ux}-1}{ux}x^{2}\nu(dx),
\]
where $\left(  e^{\i y}-1\right)  /y$ is bounded for $y\in
\mathbb{R}.$
Suppose $d\neq0.$ By (\ref{dr}), $\psi^{\prime}(0)=-\i d%
\neq0,$ and since $\psi(0)=0$ we have (\ref{u0})-$(i),$ and (\ref{u0})-$(ii)$
is obvious. Next suppose $d=0,$ i.e. $\psi^{\prime}(0)=0.$ We then have,%
\begin{align*}
\psi(u)  &  =\psi(u)-u\psi^{\prime}(0)=\psi(u)+\i ud \\
&  =\frac{u^{2}\sigma^{2}}{2}+\int_{\mathbb{R}}(1-e^{\i%
ux}+\i ux1_{\left\vert x\right\vert \leq1})\nu(dx)+\i%
u\int_{\{|x|>1\}}x\nu(dx)\\
&  =\frac{u^{2}\sigma^{2}}{2}+\int_{\{|x|\leq1\}}(1-e^{\i %
ux}+\i ux)\nu(dx)\\
&  +\int_{\{|x|>1\}}(1-e^{\i ux})\nu(dx)+\i u\int
_{\{|x|>1\}}x\nu(dx)
\end{align*}
and%
\begin{align*}
\psi^{\prime}(u)  &  =\psi^{\prime}(u)-\psi^{\prime}(0)\\
&  =u\sigma^{2}-\i\int_{\mathbb{R}}(e^{\i ux}-1_{\left\vert
x\right\vert \leq1})x\nu(dx)+\i\int_{\{|x|>1\}}x\nu(dx).
\end{align*}
If $\nu(\{|x|>1\}\cap dx)\equiv0$ we thus have%
\begin{align*}
\psi(u)  &  =\frac{u^{2}\sigma^{2}}{2}+\int_{\{|x|\leq1\}}(1-e^{\i%
ux}+\i ux)\nu(dx)\\
&  =\frac{u^{2}\sigma^{2}}{2}+u^{2}\int_{\{|x|\leq1\}}\frac{1-e^{\i %
ux}+\i ux}{\left(  ux\right)  ^{2}}x^{2}\nu(dx)
\end{align*}
and we observe that%
\[
\operatorname{Re}\left(  1-e^{\i ux}+\i ux\right)
=1-\cos(ux)\geq0
\]
so in particular $\operatorname{Re}\psi(u)\gtrsim u^{2}$ while $\left\vert
\psi(u)\right\vert \lesssim u^{2}.$ Hence (\ref{u01})-$(i)$ is shown. Then,
\begin{align*}
\psi^{\prime}(u)  &  =u\sigma^{2}-\i\int_{\{|x|\leq1\}}%
(e^{\i ux}-1)x\nu(dx)\\
&  =u\sigma^{2}-\i u\int_{\{|x|\leq1\}}\frac{e^{\i ux}%
-1}{ux}x^{2}\nu(dx)
\end{align*}
and note again that $\left(  e^{\i y}-1\right)  /y$ is bounded, hence
we have (\ref{u01})-$(ii).$ Finally, \ if $d=0$ and $\nu(\{|x|>1\}\cap dx)\neq0,$ let us write%
\begin{align*}
\psi(u)  &  =\frac{u^{2}\sigma^{2}}{2}+u^{2}\int_{\{|x|\leq1\}}\frac
{1-e^{\i ux}+\i ux}{\left(  ux\right)  ^{2}}x^{2}\nu(dx)\\
&  +\int_{\{|x|>1\}}(1-\cos(ux))\nu(dx)+\i\int_{\{|x|>1\}}\left(
ux-\sin(ux)\right)  \nu(dx)
\end{align*}
where%
\[
0\leq\int_{\{|x|>1\}}\left(  ux-\sin(ux)\right)  \nu(dx)\leq u\int
_{\{|x|>1\}}x\nu(dx)\lesssim u,
\]
but due to dominated convergence also%
\[
\int_{\{|x|>1\}}\left(  ux-\sin(ux)\right)  \nu(dx)=u\int_{\{|x|>1\}}%
x\nu(dx)+o(1).
\]
Hence,%
\[
\int_{\{|x|>1\}}\left(  ux-\sin(ux)\right)  \nu(dx)\asymp u,\text{
\ \ }u\downarrow0,
\]
and from this (\ref{u02})-$(i).$ For the derivative we have,%
\begin{align*}
\psi^{\prime}(u)  &  =u\sigma^{2}-\i\int_{\mathbb{R}}%
(e^{\i ux}-1_{\left\vert x\right\vert \leq1})xv(dx)+\i %
\int_{\{|x|>1\}}x\nu(dx)\\
&  =u\sigma^{2}-\i u\int_{\{|x|\leq1\}}\frac{e^{\i ux}%
-1}{ux}x^{2}\nu(dx)-\i\int_{\{|x|>1\}}\left(  e^{\i%
ux}-1\right)  x\nu(dx)\\
&  =o(1),\text{ \ \ \ }u\downarrow0,
\end{align*}
by similar arguments, i.e. (\ref{u02})-$(ii).$
\end{proof}

\begin{lem}
\label{lemma_gamma_asymp} For any $\alpha\geq-2,$ there exist positive
constants $C_{1}$ and $C_{2}(\alpha)$ such that uniformly for $\left\vert
\beta\right\vert \geq2,$
\begin{align}
\label{gamma_asymp}C|\beta|^{\alpha-1/2}e^{-|\beta|\pi/2}\leq\left\vert
\Gamma(\alpha+\i\beta)\right\vert \leq C_{\alpha}|\beta|^{\alpha
-1/2}e^{-|\beta|\pi/2}.
\end{align}

\end{lem}

\begin{cor}
\label{cor_integ_gamma} For all $0<\alpha<1/2$ and all $U>2,$ it holds%
\begin{equation}
\label{integ_gamma}\int_{-U}^{U}\frac{d\beta}{\left\vert \Gamma(\alpha+\i
\beta)\right\vert }\leq CU^{1/2-\alpha}e^{U\pi/2}%
\end{equation}
for a constant $C>0.$ For $\alpha>1/2,$ we have
\begin{equation}
\label{integ_gamma1}\int_{-U}^{U}\frac{d\beta}{\left\vert \Gamma(\alpha
+\i\beta)\right\vert }\leq C_{1}(\alpha)+C_{2}e^{U\pi/2}%
\end{equation}
where $C_{2}$ does not depend on $\alpha.$
\end{cor}

\bibliographystyle{plain}
\bibliography{est_subbm_bibliography}

\end{document}